\documentclass[10pt]{amsart}
\usepackage[margin=1in]{geometry}
\usepackage{amsmath,amssymb,amsthm,mathrsfs,mathtools}
\usepackage{tikz-cd}
\usepackage{booktabs,array}
\usepackage{enumitem}
\usepackage{xcolor}
\usepackage[colorlinks=true,linkcolor=blue!50!black,citecolor=green!40!black,urlcolor=blue!50!black]{hyperref}

\numberwithin{equation}{section}
\setlist{itemsep=1pt,topsep=3pt}
\newtheorem{theorem}{Theorem}[section]
\newtheorem{proposition}[theorem]{Proposition}
\newtheorem{lemma}[theorem]{Lemma}
\newtheorem{corollary}[theorem]{Corollary}
\newtheorem{conjecture}[theorem]{Conjecture}
\newtheorem{mainthm}[theorem]{Theorem}
\newtheorem{maincor}[theorem]{Corollary}
\theoremstyle{definition}
\newtheorem{definition}[theorem]{Definition}
\newtheorem{example}[theorem]{Example}
\newtheorem{remark}[theorem]{Remark}

\newcommand{\Db}{\mathrm{D}^{\mathrm{b}}}
\newcommand{\Dqc}{\mathrm{D}_{\mathrm{qc}}}
\newcommand{\Perf}{\mathrm{Perf}}
\newcommand{\Ku}{\mathcal{K}u}
\newcommand{\cO}{\mathcal{O}}
\newcommand{\cA}{\mathcal{A}}
\newcommand{\cB}{\mathcal{B}}
\newcommand{\cC}{\mathcal{C}}
\newcommand{\cE}{\mathcal{E}}
\newcommand{\cH}{\mathcal{H}}
\newcommand{\cK}{\mathcal{K}}
\newcommand{\cM}{\mathcal{M}}
\newcommand{\cT}{\mathcal{T}}
\newcommand{\PP}{\mathbb{P}}
\newcommand{\CC}{\mathbb{C}}
\newcommand{\ZZ}{\mathbb{Z}}
\newcommand{\RHom}{\operatorname{RHom}}
\newcommand{\Hom}{\operatorname{Hom}}
\newcommand{\REnd}{\operatorname{REnd}}
\newcommand{\RG}{\mathrm{R}\Gamma}
\newcommand{\id}{\mathrm{id}}
\newcommand{\Fun}{\operatorname{Fun}}

\newcommand{\Pic}{\operatorname{Pic}}
\newcommand{\thick}{\operatorname{thick}}
\newcommand{\Cone}{\operatorname{Cone}}
\newcommand{\Spec}{\operatorname{Spec}}
\newcommand{\Gr}{\operatorname{Gr}}
\newcommand{\HH}{\mathrm{HH}}
\newcommand{\HS}{\mathrm{HS}}
\newcommand{\Sym}{\operatorname{Sym}}
\newcommand{\Proj}{\operatorname{Proj}}

\newcommand{\Rot}{\mathsf{R}}

\title{Categorical Torelli theorem for smooth Fano varieties}
\author{Xun Lin}
\address{School of Science and Engineering, The Chinese University of Hong Kong, Shenzhen, Shenzhen, Guangdong 518172, China}
\email{linxun@cuhk.edu.cn}
\author{Shizhuo Zhang}
\address{School of Mathematics, Sun Yat-sen University, Guangzhou 510275, P.~R.~China}
\email{zhangshzh28@mail.sysu.edu.cn}
\date{October 2026}
\subjclass[2020]{Primary 14F08; Secondary 14J45, 14N05, 18G70}
\keywords{Categorical Torelli theorem, Kuznetsov component, rotation functor, noncommutative divisor, projective duality, Fano variety, Perry's conjecture, hypersurface, prime Fano threefold}

\DeclareMathOperator{\PsPerf}{PsPerf}
\DeclareMathOperator{\Sg}{Sg}
\providecommand{\Sing}{\operatorname{Sing}}
\begin{document}

\begin{abstract}
Let $X$ be a smooth Fano variety and $H$ an ample line bundle with $-K_X=rH$ and $r\le\dim X-2$. We prove that the appropriately defined Kuznetsov component of $X$ with respect to $H$, together with its rotation functor, determines $(X,H)$ when $H$ is very ample, and more generally when $H$ is globally generated and the image of the morphism defined by $|H|$ has a hypersurface as projective dual. As a consequence we prove Perry's conjecture, which predicts that a smooth Fano variety $X$ is determined up to isomorphism by the orthogonal complement $\cO_X^{\perp}$ of its structure sheaf in $\Db(X)$, for all smooth Fano varieties of dimension at least three with very ample anticanonical bundle, and more generally with globally generated anticanonical bundle whose image has a hypersurface as dual. We also prove it when a multiple $-qK_X$ is very ample and the corresponding embedding is cut out by equations of degree at most $(\dim X-1)/q$.

The idea is to read off $X$ from its hyperplane sections. We show that the Kuznetsov component with its rotation functor contains the residual categories of all hyperplane sections of $X$, singular ones included. The sections whose categories are not smooth form the projective dual of $X$, and classical projective duality recovers $X$. When $H$ is not very ample, the nodes of a single tangent hyperplane section, defined over a function field and detected by its singularity category, recover $X$ as a finite cover of its image. Applications include Fano hypersurfaces, all prime Fano threefolds of index one and genus at least three, cyclic covers of quadrics, Fano complete intersections and Gushel--Mukai varieties.
\end{abstract}

\maketitle
\setcounter{tocdepth}{1}
\tableofcontents

\section{Introduction}

\subsection{What does a derived category remember?}
Derived categories of coherent sheaves have become one of the central structures of algebraic geometry. They are remarkably rich invariants: they encode Hodge-theoretic and motivic information, moduli spaces of sheaves, and, through homological mirror symmetry, symplectic geometry. Yet they do not always determine the variety; an abelian variety and its dual, for instance, have equivalent derived categories \cite{Muk81}. A basic question is therefore how much of a variety its derived category, or a piece of it, remembers.

For Fano varieties the answer for the whole category is complete: a smooth Fano variety is determined by its bounded derived category $\Db(X)$ \cite{BO01}. But the derived category of a Fano variety usually contains exceptional objects, such as line bundles and tautological bundles, which carry no moduli. Removing them leaves a smaller category, a \emph{Kuznetsov component}, in which the essential geometry is concentrated \cite{Kuz14}. The Kuznetsov component of a cubic fourfold is a noncommutative K3 surface, conjecturally governing its rationality \cite{Kuz10}; Kuznetsov components of Fano threefolds encode their intermediate Jacobians \cite{BT16,Per22}; and homological projective duality relates the Kuznetsov components of linear sections of projectively dual varieties \cite{Kuz07}. The \emph{categorical Torelli problem} asks whether a Kuznetsov component determines the variety. It is a noncommutative analogue of the classical Torelli theorems, in which a curve is recovered from its principally polarized Jacobian and a smooth cubic threefold from its intermediate Jacobian \cite{CG72}.

Several quite different methods have been developed for this problem. If the Kuznetsov component carries a Bridgeland stability condition, one may recover a classical moduli space attached to $X$ as a moduli space of stable objects, and then $X$ from this moduli space; this was done for cubic threefolds \cite{BMMS12} and, after the construction of stability conditions on Kuznetsov components \cite{BLMS23}, for many Fano threefolds \cite{APR22,FLZ23,JLLZ24,JLZ26}. For hypersurfaces one may instead compute natural transformations: Hochschild cohomology graded by a degree-shift functor recovers the Jacobian ring, and hence the equation \cite{HR19,Pir24,LZ25,LRZ26,Lin26}. Hodge-theoretic invariants can be read off from topological K-theory \cite{BT16,Per22}; combined with Hodge-theoretic Torelli theorems, this gives categorical Torelli theorems for cyclic covers, in particular double covers, by Dell, Jacovskis and Rota \cite{DJR23}. When the Kuznetsov component is too small one may retain part of its gluing data with the exceptional objects \cite{JLZ26,FLZ23}. Each of these approaches requires substantial input specific to the family: a stability condition and a moduli space, an explicit equation, or a Hodge-theoretic Torelli theorem.

This paper develops a uniform and purely geometric mechanism. A Kuznetsov component, together with its rotation functor, contains the categories of all hyperplane sections of $X$, singular ones included, and the singular hyperplane sections determine $X$ through projective duality. When the polarization is not an embedding, a single hyperplane section defined over a function field recovers the covering. The mechanism uses no stability conditions, no equations and no Hodge theory. In its most canonical form it proves Perry's conjecture for a large class of Fano varieties: such a variety is determined by the category obtained from $\Db(X)$ by removing a single object, its structure sheaf.

\subsection{Kuznetsov components and rotation functors}
Let $X$ be a smooth Fano variety of dimension $n$. Its Picard group is torsion free, so $-K_X=i_XH_X$ for a unique primitive class $H_X$, the \emph{fundamental divisor}, where $i_X$ is the \emph{index} of $X$ (Section~\ref{ss:Ku}). More generally, we call a pair $(H,r)$, with $H$ ample and $\omega_X^{-1}\simeq H^{\otimes r}$, a \emph{Fano polarization}; such pairs correspond to the divisors $r$ of $i_X$. By Kodaira vanishing, $\cO_X,\cO_X(H),\dots,\cO_X((r-1)H)$ is an exceptional collection, and its right orthogonal $\Ku(X,H)=\langle\cO_X,\cO_X(H),\dots,\cO_X((r-1)H)\rangle^{\perp}$ is the \emph{Kuznetsov component} of $X$ with respect to $(H,r)$, so that
\begin{equation}\label{eq:SODintro}
\Db(X)=\bigl\langle\Ku(X,H),\cO_X,\cO_X(H),\dots,\cO_X((r-1)H)\bigr\rangle.
\end{equation}
Two choices are distinguished. The \emph{canonical Kuznetsov component} $\Ku(X)=\Ku(X,H_X)$ depends on no choice; for a smooth hypersurface of degree $d\le n+1$ in $\PP^{n+1}$, $n\ge3$, it is the usual Kuznetsov component. The \emph{anticanonical} one, $\Ku(X,-K_X)=\cO_X^{\perp}$, is the residual category of the structure sheaf; it contains $\Ku(X)$, and the two agree exactly when $i_X=1$.

Each Kuznetsov component carries the \emph{rotation functor} of Kuznetsov \cite{Kuz19}, $\Rot_{X,H}=\pi\circ(-\otimes\cO_X(H))\circ i$, where $i$ is the inclusion and $\pi$ its left adjoint; explicitly,
\begin{equation}\label{eq:rotformula}
\Rot_{X,H}(F)\simeq\Cone\bigl(\RG(X,F(H))\otimes\cO_X\xrightarrow{\ \mathrm{ev}\ }F(H)\bigr).
\end{equation}
It is an autoequivalence with $\Rot_{X,H}^{\,r}\simeq\mathsf S^{-1}_{\Ku(X,H)}[n]$ (Proposition~\ref{prop:serre}). Hence for $\cO_X^{\perp}$ the rotation functor $\mathsf S^{-1}[n]$ is determined by the category, while for $r\ge2$ it is additional structure, which cannot be dropped in general: there are non-isomorphic cubic fourfolds with equivalent Kuznetsov components \cite{Per21}. We write $\Rot_X$ for the rotation functor of $\Ku(X)$, and say that an equivalence $\Phi\colon\Ku(X,H_X)\simeq\Ku(Y,H_Y)$ \emph{commutes with the rotation functors} if $\Phi\circ\Rot_{X,H_X}\simeq\Rot_{Y,H_Y}\circ\Phi$ as enhanced functors.

\subsection{The categorical Torelli theorem with rotation functors}
For a globally generated ample line bundle $H$ on $X$ we write
\[
f_H\colon X\longrightarrow Z_H\subset\PP\bigl(H^0(X,H)^\vee\bigr)
\]
for the morphism defined by the complete linear system $|H|$, with reduced image $Z_H$. Since $H$ is ample, $f_H$ is finite; if $H$ is very ample, it is the embedding. Our first main result says that the rotation functor recovers $(X,H)$ in both situations, provided in the second one that the projective dual $Z_H^\vee$ is a hypersurface.

\begin{mainthm}[Categorical Torelli theorem with rotation functors]\label{thm:A}
Let $X$ and $Y$ be smooth Fano varieties of dimension $n$, and write $\omega_X^{-1}\simeq H_X^{\otimes r}$ and $\omega_Y^{-1}\simeq H_Y^{\otimes r}$ with $r\le n-2$ and $H_X,H_Y$ globally generated. Assume that either $H_X$ and $H_Y$ are very ample, or the projective dual of $Z_{H_X}$ is a hypersurface. If $\Ku(X,H_X)\simeq\Ku(Y,H_Y)$ by an equivalence commuting with the rotation functors, then $(X,H_X)\simeq(Y,H_Y)$, by an isomorphism lying over a projective identification of $Z_{H_X}$ and $Z_{H_Y}$.
\end{mainthm}

\begin{remark}\label{rem:numerical}
When $H_X$ and $H_Y$ are very ample, the proof shows that the assumption that $X$ and $Y$ have the same dimension and the same $r$ is not needed; it follows from the conclusion (Theorem~\ref{thm:general}). For the anticanonical polarization the dimension is always determined by the category (Lemma~\ref{lem:dim}).
\end{remark}

\begin{remark}\label{rem:hypA}
\leavevmode
\begin{enumerate}[label=(\alph*)]
\item The image $Z_H$ is the image of $|H|$. For $r\ge2$ it differs from the anticanonical image; for $r=1$ the theorem concerns precisely the anticanonical morphism.
\item When $H_X$ and $H_Y$ are very ample, no condition on the dual is needed, since projective biduality holds also for dual-defective varieties, whose dual is not a hypersurface (Section~\ref{ss:duality}). For $H=-K_X$ very ample the dual is automatically a hypersurface (Lemma~\ref{lem:vadual}).
\item For the fundamental polarization, the bound $r\le n-2$ excludes $\PP^n$ and quadrics, the only Fano varieties of index at least $n$ \cite{KO73}, and del Pezzo varieties, of index $n-1$. The case $r=n-1$ is covered by the same proof whenever the natural map $H^0(X,H)\to\Hom(\id,\Rot_{X,H})$ is an isomorphism (Theorems~\ref{thm:general} and~\ref{thm:finite}); this includes cubic hypersurfaces (Corollary~\ref{cor:hypmarked}).
\item The structure sheaf may be replaced by a completely orthogonal collection of exceptional coherent sheaves of positive rank, without any local freeness or vanishing of higher Ext groups against twists (Theorems~\ref{thm:general} and~\ref{thm:finite}). When $H$ is very ample, this assumption may be weakened to the condition that $|H|$ is base point free and defines an unramified morphism (Proposition~\ref{prop:unramified}).
\item For a cyclic cover $\pi\colon X\to Q$ of a smooth quadric, the polarization $H=\pi^*\cO_Q(1)$ is not very ample, but $Z_H=Q$ has a quadric as dual, and Theorem~\ref{thm:A} applies (Corollary~\ref{cor:cyclicquadric}).
\end{enumerate}
\end{remark}

\subsection{Perry's conjecture}
Among the Kuznetsov components of $X$, the anticanonical one is the most canonical: it requires no choice of exceptional objects beyond the structure sheaf, and since its rotation functor is the shifted inverse Serre functor, it carries no extra structure. Perry conjectured that it always remembers the variety.

\begin{conjecture}[Perry]\label{conj:Perry}
Let $X$ and $Y$ be smooth Fano varieties. If $\cO_X^{\perp}\simeq\cO_Y^{\perp}$, then $X\simeq Y$.
\end{conjecture}

Let us recall what was known. In dimension one, $X\simeq\PP^1$ and $\cO_X^{\perp}\simeq\Db(\mathrm{pt})$; conversely, if $\cO_Y^{\perp}\simeq\Db(\mathrm{pt})$ for a smooth Fano variety $Y$, then the Hochschild homology of $Y$ is two-dimensional, so the Hodge numbers of $Y$ add up to two and $Y\simeq\PP^1$. Thus the conjecture is trivial in dimension one. In dimension two, $X$ is a del Pezzo surface, and for pairs of del Pezzo surfaces the conjecture follows from previous work: from the categorical Torelli theorem of Lin--Zhang for cubic surfaces \cite{LZ25}, from the work of Lin--Rampazzo--Zhang on del Pezzo surfaces of degrees one to four via Hochschild--Serre algebras \cite{LRampZ26}, and from Elagin's categorical Torelli theorem for quartic del Pezzo surfaces \cite{Ela26}; in degree at least five the surfaces are rigid and the statement is elementary (Remark~\ref{rem:dP5}). In dimension at least three only special cases were known, all of index one: smooth hypersurfaces of degree $n+1$ in $\PP^{n+1}$ \cite{LZ25}, and refined statements for prime Fano threefolds of index one which use, besides the smaller Kuznetsov component $\langle\cE_X,\cO_X\rangle^{\perp}$, where $\cE_X$ is the Mukai bundle, an additional gluing object \cite{JLLZ24,JLZ26}.

Our second main result proves the conjecture whenever the anticanonical morphism is finite onto an image with hypersurface dual; the morphism may have arbitrary degree.

\begin{mainthm}[Perry's conjecture for finite anticanonical maps]\label{thm:B}
Let $X$ and $Y$ be smooth Fano varieties of dimension at least three with globally generated anticanonical bundles, and assume that the projective dual of the anticanonical image $Z_{-K_X}$ of $X$ is a hypersurface. If $\cO_X^{\perp}\simeq\cO_Y^{\perp}$, then $X\simeq Y$, by an isomorphism lying over a projective identification of the anticanonical images. The same holds for an equivalence $\cB_X^{\perp}\simeq\cB_Y^{\perp}$, where $\cB_X,\cB_Y$ are completely orthogonal collections of exceptional coherent sheaves of positive rank, that is, collections $E_1,\dots,E_s$ with $\RHom(E_\alpha,E_\beta)=0$ for $\alpha\ne\beta$.
\end{mainthm}

\begin{maincor}[Perry's conjecture for very ample anticanonical bundles]\label{cor:vaPerry}
Let $X$ and $Y$ be smooth Fano varieties of dimension at least three with very ample anticanonical bundles. If $\cO_X^{\perp}\simeq\cO_Y^{\perp}$, then $X\simeq Y$, by an isomorphism identifying the anticanonical embeddings. The same holds for the orthogonals of collections as in Theorem~\ref{thm:B}.
\end{maincor}

\begin{remark}\label{rem:FMreplacement}
In Theorem~\ref{thm:B}, Corollary~\ref{cor:vaPerry} and Theorem~\ref{thm:C} below, the equivalence $\cO_X^{\perp}\simeq\cO_Y^{\perp}$ need not be assumed to be enhanced, or of Fourier--Mukai type. By \cite[Theorem~1.1 and Corollary~1.2]{LZ26}, the orthogonal of the structure sheaf of a smooth Fano variety has a unique dg enhancement, and every exact $\CC$-linear equivalence $\cO_X^{\perp}\simeq\cO_Y^{\perp}$ implies the existence of a Fourier--Mukai equivalence $\cO_X^{\perp}\simeq\cO_Y^{\perp}$, which can be chosen with the same action on isomorphism classes of objects; no equality of dimensions is required. Since the conclusions only depend on the existence of an equivalence, they hold for arbitrary exact $\CC$-linear equivalences of triangulated categories. The same applies to orthogonals of exceptional vector bundles.
\end{remark}

Our third main result replaces the conditions on the anticanonical system by conditions on the equations of a pluri-anticanonical embedding. For a smooth proper enhanced category $\cA$ with Serre functor $\mathsf S_\cA$, the \emph{Hochschild--Serre algebra}, introduced by Belmans, Fu and Krug under the name Hochschild--Serre cohomology \cite{BFK23} and by Lin and Zhang \cite{LZ25}, is the bigraded algebra
\[
\HS^{\bullet,\bullet}(\cA)=\bigoplus_{p,q\in\ZZ}\HS^{p,q}(\cA),\qquad\HS^{p,q}(\cA)=\Hom\bigl(\id_\cA,\mathsf S_\cA^{q}[p]\bigr),
\]
with the product $u\cdot v=\mathsf S_\cA^{q}[p](v)\circ u$ for $u\in\HS^{p,q}(\cA)$; it is invariant under equivalences \cite[Theorem~3.4]{LZ25}. For a smooth Fano variety $X$ we write $\HS^{p,q}_X=\HS^{p,q}(\cO_X^{\perp})$.

\begin{mainthm}[Perry's conjecture for very ample anticanonical multiples]\label{thm:C}
Let $q\ge1$ and $\ell\ge2$ be integers. Let $X$ and $Y$ be smooth Fano varieties with $q\ell\le\dim X-1$ and $q\ell\le\dim Y-1$, such that $-qK_X$ and $-qK_Y$ are very ample and the homogeneous ideals of the embeddings
\[
X\subset\PP\bigl(H^0(X,-qK_X)^\vee\bigr),\qquad Y\subset\PP\bigl(H^0(Y,-qK_Y)^\vee\bigr)
\]
are generated in degrees at most $\ell$. If $\cO_X^{\perp}\simeq\cO_Y^{\perp}$, then $X\simeq Y$, by an isomorphism identifying these embeddings. Moreover, $X$ is recovered from $\cO_X^{\perp}$ by the presentation
\[
X\simeq\Proj\Bigl(\Sym\HS^{qn,-q}_X\Big/\bigl(\ker\bigl(\Sym^m\HS^{qn,-q}_X\to\HS^{qmn,-qm}_X\bigr):2\le m\le\ell\bigr)\Bigr),
\]
where $n=\dim X$ is determined by $\cO_X^{\perp}$ through $qn=\min\{p\in\ZZ:\HS^{p,-q}_X\ne0\}$ (Lemma~\ref{lem:dimq}).
\end{mainthm}

\begin{remark}\label{rem:complementary}
Theorems~\ref{thm:B} and~\ref{thm:C} are complementary. Theorem~\ref{thm:B} requires a globally generated anticanonical bundle with image of hypersurface dual, and no condition on equations. Theorem~\ref{thm:C} requires equations of bounded degree for a pluri-anticanonical embedding, and allows base points of $|-K_X|$. Combining them with the known cases in dimensions one and two, Perry's conjecture holds for pairs of del Pezzo surfaces and for all smooth Fano varieties of dimension at least three satisfying the hypotheses of Theorem~\ref{thm:B} or of Theorem~\ref{thm:C}.
\end{remark}

\subsection{Consequences}
For a smooth hypersurface $X\subset\PP^{n+1}$ of degree $d$ with $4\le d\le n+1$, the polarization $H=\cO_X(1)$ is very ample with $r=n+2-d\le n-2$, so Theorem~\ref{thm:A} gives the categorical Torelli theorem with rotation functors. For cubic hypersurfaces of dimension $n\ge4$ we have $r=n-1$, outside the range of Theorem~\ref{thm:A}, but the same argument applies because $H^0(X,\cO_X(1))\to\Hom(\id,\Rot_X)$ is an isomorphism (Remark~\ref{rem:hypA}(c)). For del Pezzo threefolds of degree $2\le d\le4$, which include cubic threefolds, and for prime Fano threefolds of index one and genus $g\ge6$, the usual Kuznetsov component $\cK_X$ is smaller than $\cO_X^{\perp}$: one has $\cO_X^{\perp}=\langle\cK_X,E_X\rangle$ for one further exceptional object $E_X$. The way $E_X$ is attached to $\cK_X$ is recorded by a single object of $\cK_X$, namely $i^!E_X$, where $i^!$ is the right adjoint of the inclusion $i\colon\cK_X\to\Db(X)$; an equivalence $\cK_X\simeq\cK_Y$ which matches these objects therefore extends to an equivalence $\cO_X^{\perp}\simeq\cO_Y^{\perp}$ (Corollary~\ref{cor:gluingobject}), and Corollary~\ref{cor:vaPerry} applies. For del Pezzo threefolds every equivalence matches these objects \cite{FLZ23}, so $\cK_X$ alone determines $X$. Finally, Theorems~\ref{thm:A} and~\ref{thm:B} for globally generated polarizations which are not very ample give cyclic covers of quadrics and of projective spaces, among them the hyperelliptic prime Fano threefolds of genus three.

\begin{maincor}\label{cor:C}
The following categorical Torelli theorems hold.
\begin{enumerate}[label=\textup{(\alph*)}]
\item \textup{(Smooth hypersurfaces.)} Let $n\ge3$ and $3\le d\le n+1$. Smooth hypersurfaces $X,Y\subset\PP^{n+1}$ of degree $d$ whose Kuznetsov components are equivalent compatibly with the rotation functors are projectively equivalent \cite{HR19,Pir24,BMMS12}; if $\gcd(n+2,d)=1$, every equivalence $\Ku(X)\simeq\Ku(Y)$ suffices \cite{LZ25}. For $d=n+1$, every smooth Fano variety $Y$ of dimension at least three with globally generated anticanonical bundle and $\cO_Y^{\perp}\simeq\cO_X^{\perp}$ is isomorphic to $X$, and its anticanonical morphism is an embedding as a hypersurface projectively equivalent to $X$.
\item \textup{(Prime Fano threefolds of index one.)} A smooth prime Fano threefold $X$ of index one and genus $g\ge3$ is determined by $\cO_X^{\perp}$, even among all smooth Fano varieties of dimension at least three with globally generated anticanonical bundle; this includes the quartic threefolds and the hyperelliptic threefolds of genus three. For genus $g\in\{6,7,8,9,10,12\}$, with the decomposition $\Db(X)=\langle\cK_X,\cE_X,\cO_X\rangle$, an equivalence $\cK_X\simeq\cK_Y$ sending $i^!\cE_X$ to $i^!\cE_Y$ implies $X\simeq Y$; this is the refined categorical Torelli theorem of \cite{JLZ26}, and of \cite{JLLZ24} for Gushel--Mukai threefolds.
\item \textup{(Further examples.)} Theorems~\ref{thm:A} and~\ref{thm:B} apply to the following varieties; in \textup{(iv)} one first extends the equivalence from $\Ku(Y_d)$ to $\cO_{Y_d}^{\perp}$, as explained above.
\begin{enumerate}[label=\textup{(\roman*)}]
\item Smooth Fano complete intersections $X\subset\PP^{n+c}$ of type $(d_1,\dots,d_c)$ with all $d_j\ge2$, $n\ge3$ and $\sum_j(d_j-1)\ge3$: the pair $(\Ku(X),\Rot_X)$ determines $X$ up to projective equivalence, and if $\sum_j(d_j-1)=n$, so that $X$ has index one, $\Ku(X)=\cO_X^{\perp}$ alone suffices. Examples are complete intersections of a quadric and a hypersurface of degree $d\ge3$, and the index-one complete intersections of types $(4)$, $(3,2)$, $(2,2,2)$ in dimension three, $(5)$, $(4,2)$, $(3,3)$, $(3,2,2)$, $(2,2,2,2)$ in dimension four, and so on.
\item Ordinary Gushel--Mukai varieties of dimension three, four and five, by $(\Ku(X),\Rot_X)$.
\item Smooth hypersurfaces $X\subset\PP^a\times\PP^b$ of bidegree $(d_1,d_2)$ with $1\le d_1\le a$, $1\le d_2\le b$ and $a+b\ge4$, by $\cO_X^{\perp}$; for instance $(2,2)$-threefolds in $\PP^2\times\PP^2$.
\item Del Pezzo threefolds $Y_d$ of Picard rank one and degree $d\in\{2,3,4\}$, by $\Ku(Y_d)$ alone \cite{FLZ23}, by gluing $\Ku(Y_d)$ to $\cO_{Y_d}^{\perp}$ and applying Corollary~\ref{cor:vaPerry}.
\item Cyclic covers $\pi\colon X\to Q$ of degree $d\ge2$ of a smooth quadric $Q$ of dimension $n\ge3$, branched along a smooth divisor in $|\cO_Q(db)|$ with $b\ge2$ and $r=n-(d-1)b\ge1$: by $(\Ku(X,H),\Rot_{X,H})$ for the polarization $H=\pi^*\cO_Q(1)$, which is not very ample, and by $\cO_X^{\perp}$ alone. Examples are the hyperelliptic prime Fano threefolds of genus three and the triple covers of a quadric fivefold branched along a sextic section, whose anticanonical morphisms have degrees two and three.
\item Cyclic covers of $\PP^n$, $n\ge3$, of degree $d\ge2$, branched along a smooth hypersurface of degree $db$, with $n+1-(d-1)b\ge2$, by $\cO_X^{\perp}$; for instance the double cover of $\PP^6$ branched along a decic, whose anticanonical morphism is two-to-one onto the Veronese variety $\nu_2(\PP^6)$.
\end{enumerate}
\end{enumerate}
\end{maincor}

Corollary~\ref{cor:C} is proved in Sections~\ref{sec:hyp}--\ref{sec:examples}: part (a) is Corollaries~\ref{cor:hypmarked}, \ref{cor:hypcoprime} and~\ref{cor:hypindexone}; part (b) is Corollaries~\ref{cor:primePerry} and~\ref{thm:refined}; and in part (c), items (i)--(iv) are Theorem~\ref{thm:CI}, Propositions~\ref{prop:GM} and~\ref{prop:products} and Theorem~\ref{thm:dP3}, item (v) is Corollaries~\ref{cor:cyclicquadric} and~\ref{cor:hyperelliptic} and Example~\ref{ex:triplecover}, and item (vi) is Proposition~\ref{prop:veronese} and Example~\ref{ex:decic}. Table~\ref{tab:comparison} compares these results with the literature.

\subsection{Hyperplane sections inside the category}
The key to Theorem~\ref{thm:A} is that the Kuznetsov component with its rotation functor contains the residual categories of all hyperplane sections of $X$. Multiplication by a section $s\in H^0(X,H)$ induces a natural transformation $\id\to\Rot_{X,H}$. To such a transformation we attach a \emph{noncommutative divisor}: an $A_\infty$ algebra which extends an algebra model of $\Ku(X,H)$ by the bimodule of the inverse rotation, with connecting map $s$ (Section~\ref{sec:ncdiv}). For a divisor $D\in|H|$ put
\[
\Ku(D,H)=\langle\cO_D,\cO_D(H),\dots,\cO_D((r-2)H)\rangle^{\perp}\subset\Perf(D),
\]
with $\Ku(D,H)=\Perf(D)$ if $r=1$.

\begin{theorem}\label{thm:sections}
Let $X$ be a smooth Fano variety of dimension $n$, and $(H,r)$ a Fano polarization with $r\le n-2$; $H$ is only assumed to be ample. Then the pair $(\Ku(X,H),\Rot_{X,H})$ determines the following data, compatibly with all equivalences commuting with the rotation functors:
\begin{enumerate}[label=\textup{(\arabic*)}]
\item the linear system $|H|$, through a canonical isomorphism $\PP\Hom(\id,\Rot_{X,H})\simeq|H|$;
\item for every divisor $D\in|H|$, possibly singular or nonreduced, the category $\Ku(D,H)$, as the category of perfect modules over the unique noncommutative divisor with section $[D]$; the same holds over every field extension of $\CC$, for divisors of sections of $H$ defined over that field;
\item the locus of singular divisors in $|H|$, as the set of $D$ for which $\Ku(D,H)$ is not homologically smooth.
\end{enumerate}
If $H$ is very ample, the locus in \textup{(3)} is the projective dual variety $X^\vee$. In particular, for every smooth Fano variety $X$ of dimension at least three, $\cO_X^{\perp}$ determines the perfect derived categories $\Perf(D)$ of all anticanonical divisors $D$, compatibly with all equivalences $\cO_X^{\perp}\simeq\cO_Y^{\perp}$ with $\dim Y=\dim X$.
\end{theorem}

\subsection{Outline of the proof and the new ingredients}
We first explain the proof of Theorem~\ref{thm:C}, which is the most direct one: it recovers the equations of $X$ from the category, using the Hochschild--Serre algebra and a computation of Fourier--Mukai kernels. We then explain Theorems~\ref{thm:A} and~\ref{thm:B}, which avoid equations altogether and recover $X$ from its hyperplane sections.

\emph{Theorem~\ref{thm:C}: recovering the equations.} The idea is to find the section ring of $X$, and hence its equations, inside the category $\cO_X^{\perp}$. The Hochschild--Serre algebra $\HS^{\bullet,\bullet}_X$ is an algebra attached to $\cO_X^{\perp}$ alone, so it is unchanged under equivalences. A section of $-mK_X$ gives, by multiplication, a natural transformation from the identity to a power of the Serre functor, that is, an element of $\HS^{mn,-m}_X$. This defines maps
\[
\theta_m\colon H^0(X,-mK_X)\longrightarrow\HS^{mn,-m}_X,
\]
which respect products (Lemma~\ref{lem:mult}). If the $\theta_m$ were isomorphisms in all degrees, the Hochschild--Serre algebra would contain the anticanonical ring of $X$, and $X$ would be its $\Proj$. For hypersurfaces this is the mechanism of \cite{HR19,LZ25,LRZ26,Lin26}, where the algebra is identified with a Jacobian ring and the equation is recovered.

For a general Fano variety we control the maps $\theta_m$ by computing the Fourier--Mukai kernels of the functors involved. Following Pirozhkov \cite{Pir24}, who did this for the degree-shift functor of a hypersurface, we write the relevant functors as compositions of mutation functors and compute their kernels on $X\times X$. For an arbitrary smooth Fano variety this shows that $\theta_m$ is bijective for $m\le n-2$ and injective for $m\le n-1$ (Proposition~\ref{prop:kernel}). Consequently, for $q\le n-2$, the products in the Hochschild--Serre algebra of the elements of $\HS^{qn,-q}_X$ behave exactly like the products of sections of $-qK_X$, as long as the degree $qm$ is at most $n-1$. In particular the category sees which polynomials of degree at most $(n-1)/q$ vanish on $X$ in the embedding by $|-qK_X|$. When these equations generate the ideal of $X$, they cut out $X$, and $X$ is recovered (Appendix~\ref{app:equations}). In higher degrees the kernel computation no longer controls $\theta_m$, which is why Theorem~\ref{thm:C} needs a bound on the degrees of the equations.

\emph{Theorems~\ref{thm:A} and~\ref{thm:B}: hyperplane sections instead of equations.} To remove the condition on the equations, we no longer try to recover the section ring. Instead we recover the hyperplane sections of $X$ inside the category, and read off $X$ from their singularities. Let $\cA$ be the Kuznetsov component and $\Rot$ its rotation functor. Choose a classical generator of $\cA$, let $A$ be its endomorphism dg algebra, and let $P$ be the $A$-bimodule representing $\Rot^{-1}$. A natural transformation $s\colon\id\to\Rot$ is then a bimodule map $P\to A$. The argument has four steps:
\[
\begin{tikzcd}[column sep=1.6em,row sep=1.4em]
(\cA,\Rot)\arrow[r,"\text{Step 1}"] & \PP\Hom(\id,\Rot)=|H|\arrow[r,"\text{Step 2}"] & \bigl\{\Ku(D_s,H)\bigr\}_{[s]\in|H|}\arrow[d,"\text{Step 3}"]\\
X\arrow[r,leftarrow,"\text{Step 4}"'] & Z_H\arrow[r,leftarrow,"\text{biduality}"'] & \text{singular sections}\supset Z_H^\vee
\end{tikzcd}
\]

\emph{Step 1: sections.} By the kernel computation, $\Hom(\id,\Rot)\simeq H^0(X,H)$ when $r\le n-2$.

\emph{Step 2: noncommutative divisors.} For a section $s$ with zero divisor $D$, the Koszul complex of $s$ gives a dg algebra $B_s=A\oplus P[1]$, an extension of $A$ by $P[1]$ with connecting map $s$, whose perfect modules form $\Ku(D,H)$. The point is that $B_s$ is determined by $s$ alone, as an $A_\infty$ algebra and not only as a bimodule, and that the same holds for sections defined over any field extension of $\CC$.

\emph{Step 3: the image.} Since $D$ is smooth exactly when $\Ku(D,H)$ is homologically smooth, the category determines the locus of singular divisors. For an embedding this is $X^\vee$, and projective biduality recovers $X$, even if the dual is defective. For a finite map whose image $Z=Z_H$ has hypersurface dual, $Z^\vee$ is a hypersurface component of the locus of singular divisors, and an incidence computation recovers $Z$ from it (Lemmas~\ref{lem:incidence} and~\ref{lem:commonimage}).

\emph{Step 4: the finite cover.} Put $K=\CC(Z)$ and $F=\CC(Z^\vee)$. The conormal correspondence identifies $F$ with a purely transcendental extension of $K$. The generic tangent hyperplane, defined over $F$, pulls back to a divisor on $X_F$ whose singularities are ordinary double points, lying exactly over the generic contact point. Its singularity category, which is determined by its residual category, splits into indecomposable local factors, one for each node, and the Galois action on these factors recovers the finite étale $F$-algebra $\CC(X)\otimes_KF$. Taking the algebraic closure of $K$ in it recovers $\CC(X)$, and normalization of $Z$ in $\CC(X)$ recovers $X$.

Four ingredients of this argument are new in the context of categorical Torelli problems.
\begin{enumerate}[label=(\roman*)]
\item \emph{The full kernel estimate.} The Hochschild--Serre approach, including Pirozhkov's computation, only uses natural transformations $\id\to\Rot^m$ in cohomological degree zero. Step~2 needs more: the higher products of $B_s$ are controlled by the complexes $\RHom(\id,\Rot^m)$ in the negative degrees $1-2m$ and $2-2m$. We therefore extend the kernel computation to the whole complex, and show that $\RHom(\id,\Rot^m)$ has no cohomology below degree $\min\{0,n-r-m\}$ for every $m\ge0$ (Proposition~\ref{prop:kernel}). The proof expands the rotation into a cube of mutation kernels and applies duality along the diagonal; it uses only Kodaira vanishing. The bound decreases by one with each power of $\Rot$, whereas the relevant obstruction degrees decrease by two; this difference is what makes the argument work in all weights, and is the reason why no bound on the degrees of equations appears in Theorems~\ref{thm:A} and~\ref{thm:B}.
\item \emph{Uniqueness of noncommutative divisors.} Following Seidel \cite{Sei17}, we view divisor structures on $A\oplus P[1]$ as Maurer--Cartan elements in a deformation complex filtered by weight. Seidel's uniqueness criterion requires the vanishing of all negative cohomology of these complexes, which need not hold here; we refine it so that only the groups in degrees $-1$, $1-2m$ and $2-2m$ have to vanish, and these are exactly the groups killed by the full kernel estimate. Hence a noncommutative divisor is determined, up to $A_\infty$ isomorphism, by its section (Proposition~\ref{prop:unique}). Together with a relative homological perturbation argument, this allows us to transport divisor algebras along any equivalence commuting with the rotation functors.
\item \emph{A categorical discriminant.} We prove that the residual category of every hyperplane section, including singular and nonreduced ones, is generated by the restriction of a generator of $\cA$, via compact generation in the unbounded derived category (Proposition~\ref{prop:generation}). Homological smoothness then detects singular sections, so the category recovers the projective dual variety, and classical projective duality replaces the Jacobian-ring, Hodge-theoretic or moduli-theoretic inputs of earlier approaches.
\item \emph{Tangent divisors over function fields.} For a finite polarization map, the covering is recovered from one divisor, defined over the function field $F$ of the dual of the image. Its singularity category is recovered from its residual category through pseudo-perfect modules \cite{Efi17}, and decomposes into local factors at the nodes \cite{Orl11,Dyc11}; their Galois action over $F$ is a categorical invariant, and it determines the covering field (Section~\ref{sec:finite}).
\end{enumerate}
For Theorem~\ref{thm:B} one more observation is needed: for $\cO_X^{\perp}$ the rotation functor is $\mathsf S^{-1}[n]$, so it commutes with every equivalence once the dimensions agree, and the dimension is detected by the degree-zero part of the kernel estimate (Lemma~\ref{lem:dim}). Theorem~\ref{thm:B} is then the case $(H,r)=(-K_X,1)$ of the finite-map theorem, and Corollary~\ref{cor:vaPerry} the corresponding case of the embedded one.

For applications without a rotation marking we add a recognition step: the rotation functor is recovered from the Serre functor, from the periodicity of graded matrix factorizations, or from gluing data. Refined theorems for smaller Kuznetsov components follow by an enhanced gluing argument, which reconstructs $\cO_X^{\perp}$ from the smaller component and its gluing data.

\subsection{Organization and logical structure}
Section~\ref{sec:prelim} collects background on Fano polarizations, Kuznetsov components and rotation functors, kernels, homological smoothness, projective duality and gluing. Section~\ref{sec:kernel} proves the full kernel estimate. Section~\ref{sec:ncdiv} develops noncommutative divisors and proves their uniqueness. Section~\ref{sec:sections} constructs the divisor algebra of every hyperplane section, over any field extension, and proves generation. Section~\ref{sec:main} identifies the categorical discriminant and proves Theorem~\ref{thm:sections}, Theorem~\ref{thm:A} for very ample $H_X,H_Y$, and Corollary~\ref{cor:vaPerry}. Section~\ref{sec:finite} treats finite polarization maps and completes the proofs of Theorems~\ref{thm:A} and~\ref{thm:B}. Sections~\ref{sec:hyp}--\ref{sec:examples} prove Corollary~\ref{cor:C}: hypersurfaces in Section~\ref{sec:hyp}, prime Fano threefolds of index one in Section~\ref{sec:prime}, and further examples, cyclic covers and a comparison with the literature in Section~\ref{sec:examples}. Appendix~\ref{app:equations} proves Theorem~\ref{thm:C}; it uses Sections~\ref{sec:prelim} and~\ref{sec:kernel}, but not Sections~\ref{sec:ncdiv}--\ref{sec:finite}. Table~\ref{tab:results} summarizes the hypotheses of the main results and the parts of the paper on which their proofs rely.

\begin{table}[ht]
\centering
\small
\renewcommand{\arraystretch}{1.2}
\begin{tabular}{>{\raggedright\arraybackslash}p{2.4cm}>{\raggedright\arraybackslash}p{2.6cm}>{\raggedright\arraybackslash}p{4.6cm}>{\raggedright\arraybackslash}p{3.4cm}}
\toprule
Result & Category & Hypotheses & Proof\\
\midrule
Theorem~\ref{thm:sections} & $\Ku(X,H)$ with $\Rot_{X,H}$ & $H$ ample, $r\le n-2$ & Sections~\ref{sec:kernel}--\ref{sec:main}\\
Theorem~\ref{thm:A}, $H$ very ample & $\Ku(X,H)$ with $\Rot_{X,H}$ & $H$ very ample, $r\le n-2$ & Theorem~\ref{thm:general}\\
Theorem~\ref{thm:A}, $H$ globally generated & $\Ku(X,H)$ with $\Rot_{X,H}$ & $H$ globally generated, $Z_{H_X}^\vee$ a hypersurface, $r\le n-2$ & Theorem~\ref{thm:finite}\\
Theorem~\ref{thm:B} & $\cO_X^{\perp}$, or $\cB_X^{\perp}$ & $-K$ globally generated, $Z_{-K_X}^\vee$ a hypersurface, $n\ge3$ & Lemma~\ref{lem:dim} and Theorem~\ref{thm:finite}\\
Corollary~\ref{cor:vaPerry} & $\cO_X^{\perp}$, or $\cB_X^{\perp}$ & $-K$ very ample, $n\ge3$ & Lemma~\ref{lem:dim} and Theorem~\ref{thm:general}\\
Theorem~\ref{thm:C} & $\cO_X^{\perp}$ & $-qK$ very ample, ideal generated in degrees $\le\ell$, $q\ell\le n-1$ & Section~\ref{sec:kernel} and Appendix~\ref{app:equations}\\
Corollary~\ref{cor:C} & various & see Sections~\ref{sec:hyp}--\ref{sec:examples} & Theorems~\ref{thm:A}, \ref{thm:B}, gluing (Section~\ref{ss:gluing}), cited inputs\\
\bottomrule
\end{tabular}
\caption{The main results, their hypotheses, and the parts of the paper on which their proofs rely. Here $n=\dim X$, and $(H,r)$ is a Fano polarization. Global generation is required on both varieties, the hypersurface-dual condition on $X$ only.}
\label{tab:results}
\end{table}

\subsection*{Conventions}
We work over $\CC$. All categories, functors, adjunctions and natural transformations are enhanced, and an equivalence means a $\CC$-linear Morita equivalence of the standard dg enhancements. Spaces of natural transformations are cohomology groups of enhanced functor complexes, not the center of a triangulated homotopy category. For a semiorthogonal decomposition $\langle\mathcal U,\mathcal V\rangle$ we have $\Hom(\mathcal V,\mathcal U)=0$, and $E^{\perp}$ consists of the objects $F$ with $\RHom(E,F)=0$. For a vector space $W$, $\PP(W)$ is the space of lines in $W$; thus a variety embedded by a linear system $V=H^0(X,H)$ lies in $\PP(V^\vee)$, and its hyperplane sections are parametrized by $\PP(V)$.

\subsection*{Origins and development of the project}
This project grew out of our refined categorical Torelli theorems for Gushel--Mukai threefolds \cite{JLLZ24} and for prime Fano threefolds of index one and genus $g\ge6$ \cite{JLZ26}. For $g\le5$ the relevant Kuznetsov component is $\cO_X^{\perp}$, and $X$ is a Bridgeland moduli space of stable objects in the class of the ideal sheaf $I_x$ of a point; reconstructing $X$ from the abstract category then requires showing that an equivalence sends each $I_x$, after normalization, to some $I_{x'}$. This would follow from a Serre-invariance property of stability conditions on $\cO_X^{\perp}$, which is not available. For $g\ge6$, on the other hand, the refined data already give an equivalence $\cO_X^{\perp}\simeq\cO_{X'}^{\perp}$ by gluing. Both observations suggested that a more uniform argument should exist.

The Hochschild--Serre algebra method of \cite{LZ25,LRZ26} for smooth and weighted hypersurfaces gives cases of Perry's conjecture in index one, but it left open two basic index-one examples, the complete intersections of a quadric and a cubic in $\PP^5$ ($g=4$) and of three quadrics in $\PP^6$ ($g=5$), whose categorical Torelli problems are exactly cases of Perry's conjecture. Inspired by Pirozhkov's Fourier--Mukai kernel computations \cite{Pir24}, we combined them with the Hochschild--Serre algebra and proved Perry's conjecture for Gushel--Mukai threefolds in July 2026; all work up to this point was done without the assistance of generative AI. The same method then applied to prime Fano threefolds of genus eight, to Verra threefolds and, more generally, to Fano varieties whose anticanonical embedding is defined by quadrics; we began writing the paper and informed several experts.

We subsequently used GPT-5.6. Several exchanges led to the extension to pluri-anticanonical embeddings in Theorem~\ref{thm:C}: on both sides, $-qK$ is very ample and the ideal of the embedding by $|-qK|$ is generated in degrees at most $\ell$, for fixed integers $\ell\ge2$ and $q$ with $q\ell\le\dim X-1$. Since the uniform estimates require dimension at least three, we treated del Pezzo surfaces separately, where our earlier Hochschild--Serre method suffices in the nonrigid cases \cite{LRampZ26} (see also \cite{Ela26}). In parallel, inspired by \cite{HR19,Pir24}, we developed a singularity-category approach \cite{LZZ26}, which recovers the equation of a hypersurface from the Hochschild cohomology of the singularity category of its completed affine cone; under the numerical hypotheses there, it gives categorical Torelli theorems for Fano hypersurfaces and, through the full graded matrix factorization category, for hypersurfaces of general type, as well as further index-one cases of Perry's conjecture.

In early September 2026, after a dozen or so exchanges with ChatGPT (Astra), we obtained the proof of the very ample case of Theorem~\ref{thm:B}, that is, of Corollary~\ref{cor:vaPerry}, together with strengthenings of its hypotheses; for instance, very ampleness can be replaced by the condition that the anticanonical system is base point free and defines an unramified morphism (Proposition~\ref{prop:unramified}). We immediately saw that the argument applies to categorical Torelli problems with rotation functors, which led to Theorem~\ref{thm:A}: it recovers many of our earlier categorical Torelli theorems from a single statement and contains Perry's conjecture for very ample anticanonical bundles. We did not post the manuscript then, because we wanted a fuller conceptual understanding of the ideas behind the proof developed with ChatGPT, and our main effort was directed towards removing the very ampleness assumption altogether. In the present paper the very ampleness of $-K_X$ is replaced by the hypothesis of Theorem~\ref{thm:B}: $-K_X$ is globally generated and the projective dual of the anticanonical image is a hypersurface, so that the anticanonical morphism may be a finite cover of arbitrary degree. On October 16, 2026, Ziqi Liu drew our attention to Ma's preprint \cite{Ma26}, whose main proof structure and central lemmas agree closely with ours in the very ample case, although several technical implementations differ. This prompted us to make the manuscript public before it reached its intended form. We are responsible for the mathematical statements and proofs in the final text.

\subsection*{Acknowledgements}
We thank Chunyi Li and Zhiyu Liu for useful conversations during the early stages of this work. We are grateful to Chunyi Li for sharing interesting ideas on Huybrechts's conjecture for cubic fourfolds, and to Zhiyu Liu for pointing out that the techniques of \cite{LLL+26} give a new proof of the Bondal--Orlov reconstruction theorem and might be useful for Perry's conjecture. We thank Ziqi Liu for bringing Ma's preprint \cite{Ma26} to our attention. Shizhuo Zhang is supported by the Sun Yat-sen University Starting Grant No.~34000-12256019.

\subsection*{Upcoming work}
We are preparing several further papers on categorical reconstruction; their precise statements, hypotheses and proofs will appear separately.

\emph{Quadric fibrations.} Although the cases of three quadrics in $\PP^6$ and of a quadric and a cubic in $\PP^5$ are settled here, their direct Hochschild--Serre calculations lead to further reconstruction problems. For $X=V(q_1,q_2,q_3)\subset\PP^6$, the universal quadric $\mathcal Q_X=\{\lambda_1q_1+\lambda_2q_2+\lambda_3q_3=0\}\subset\PP^6\times\PP^2$ of the net is a quadric fibration over $\PP^2$, whose even Clifford algebra gives a model of $\Ku(X)$ \cite{Kuz08}, and by Corollary~\ref{cor:quadratic} the Hochschild--Serre multiplication recovers the net itself, $\langle q_1,q_2,q_3\rangle\simeq\ker(\Sym^2\HS^{3,-1}_X\to\HS^{6,-2}_X)$. This determines $X$, the universal quadratic form, its symmetric determinantal discriminant and the Clifford algebra, up to change of coordinates; we develop this approach for broader classes of Fano varieties and residual categories of quadric fibrations.

\emph{Hybrid models.} For $X=V(f_2,f_3)\subset\PP^5$ we use the hybrid Landau--Ginzburg model $W=p_2f_2(x)+p_3f_3(x)$ on the total space of $\cO_{\PP(2,3)}(-1)^{\oplus6}$, with the weighted projective line regarded as a stack and the grading and R-charge retained \cite{BS23}. The Hochschild--Serre calculation in this relative matrix factorization setting extends that of \cite{BFK14,LZ25,LRZ26}: once one shows that equivalences preserve the intrinsic data involved, it recovers the quadratic generator up to scalar and the cubic generator modulo linear multiples of the quadric, up to projective change of coordinates. This gives categorical Torelli theorems for complete intersections without compatibility with rotation functors, beyond the scope of Theorem~\ref{thm:A}.

\emph{Nodal prime Fano threefolds.} Jointly with Daniele Faenzi, we prove categorical and refined categorical Torelli theorems for almost all deformation families of one-nodal maximally non-factorial prime Fano threefolds, under admissibility hypotheses specified for each family. The categorical input is the smooth absorption component $\widetilde{\cA}_X$, together with the right-adjoint projection of the Mukai evaluation complex for the refined statements. The component alone determines $X$ for the degree-one bridges of genus four and for the families of genera five and seven; the refined results cover genera six, eight, nine and ten and the types 12nb, 12nc and 12nd in genus twelve. For type 12na we obtain reconstruction from a Hilbert--Burch marking, whose comparison with the Mukai evaluation projection remains open; the genus-two family is not included.

\emph{Hypersurfaces.} Finally, we have settled the remaining cases of the categorical Torelli problem for Fano hypersurfaces with non-Calabi--Yau Kuznetsov components; these results and their precise numerical scope will appear in a separate paper.
\section{Preliminaries}\label{sec:prelim}

\subsection{Enhancements and generators}
A dg category is \emph{proper} if its morphism complexes have finite-dimensional total cohomology. For a scheme $Z$ we write $\Perf(Z)$ for the category of perfect complexes and $\Dqc(Z)$ for the unbounded derived category of complexes with quasi-coherent cohomology; if $Z$ is smooth and projective, $\Perf(Z)=\Db(Z)$. For projective $Z$, the category $\Dqc(Z)$ is compactly generated, its compact objects are the perfect complexes, and a compact generator classically generates $\Perf(Z)$ \cite{BVdB03}. Projecting a classical generator of $\Perf(X)$ to an admissible component gives a classical generator of that component. If $G$ is such a generator and $A=\RHom(G,G)$, enhanced Morita theory identifies the component with $\Perf(A)$ \cite{Toe07}. Functors and natural transformations are represented by derived bimodules and their morphism complexes, and all tensor products of bimodules are derived.

\subsection{Fano polarizations and Kuznetsov components}\label{ss:Ku}

We recall the standard facts we need. Let $X$ be a smooth Fano variety of dimension $n$. Kodaira vanishing gives $H^i(X,\cO_X)=0$ for $i>0$, so $\Pic(X)\simeq H^2(X,\ZZ)$, and this group is torsion free because $X$ is rationally connected \cite{Cam92,KMM92}, hence simply connected \cite[Corollary~4.18]{Deb01}. Hence $-K_X=i_XH_X$ for a unique primitive class $H_X$, which is ample, and the Fano polarizations $(H,r)$ are exactly the pairs with $r\mid i_X$ and $H=(i_X/r)H_X$.

Fix a Fano polarization $(H,r)$, so that $H$ is ample and
\begin{equation}\label{eq:pol}
\omega_X^{-1}\simeq H^{\otimes r},\qquad r\ge1,
\end{equation}
and write $\cO_X(a)=H^{\otimes a}$. Kodaira vanishing for $\cO_X(t)=\omega_X\otimes H^{\otimes(t+r)}$, together with the fact that negative powers of an ample line bundle have no sections, gives
\begin{equation}\label{eq:kodaira}
H^j(X,\cO_X(t))=0\qquad(j>0,\ t>-r),
\end{equation}
\begin{equation}\label{eq:RGamma}
\RG(X,\cO_X)=\CC,\qquad\RG(X,\cO_X(t))=0\ \ (-r<t<0),\qquad\RG(X,\cO_X(-r))=\CC[-n].
\end{equation}
Since $\RHom(\cO_X(a),\cO_X(b))=\RG(X,\cO_X(b-a))$, it follows that $\cO_X,\cO_X(1),\dots,\cO_X(r-1)$ is an exceptional collection in $\Db(X)$.

We shall work in the following slightly more general setting. Let $\cB_X=\langle E_1,\dots,E_s\rangle$ be a \emph{completely orthogonal block of positive-rank exceptional sheaves}: each $E_\alpha$ is an exceptional coherent sheaf of positive rank, and $\RHom_X(E_\alpha,E_\beta)=0$ for $\alpha\ne\beta$. Assume that $\cB_X,\cB_X(H),\dots,\cB_X((r-1)H)$ is semiorthogonal, and put
\begin{equation}\label{eq:setting}
\Perf(X)=\bigl\langle\cA_X,\cB_X,\cB_X(H),\dots,\cB_X((r-1)H)\bigr\rangle.
\end{equation}
The basic example is $\cB_X=\langle\cO_X\rangle$, for which $\cA_X=\Ku(X,H)$. For $H=-K_X$ we have $r=1$, and $\cA_X$ is the orthogonal $\cB_X^{\perp}$ of an arbitrary completely orthogonal collection of positive-rank exceptional sheaves. Since a coherent sheaf on the smooth variety $X$ is perfect, all duals and restrictions below are derived.

Let $i=i_X$ be the inclusion of $\cA_X$ and $\pi=\pi_X$ its left adjoint. The \emph{rotation functor} is
\begin{equation}\label{eq:rot}
\Rot=\Rot_X=\pi\circ T\circ i,\qquad T=(-)\otimes H.
\end{equation}
Write $L_a$ for the left mutation through the block $\cB_X(aH)$. Since the block is completely orthogonal, its projection kernel is $\bigoplus_\alpha E_\alpha^\vee(-aH)\boxtimes E_\alpha(aH)$, and $L_a$ is given by the cone of the evaluation map from this kernel to $\cO_\Delta$.

\begin{lemma}\label{lem:mutation}
There are isomorphisms of functors
\begin{equation}\label{eq:mutation}
i\pi=L_0L_1\cdots L_{r-1},\qquad i\Rot=L_0Ti,\qquad i\Rot^m\pi=L_0L_1\cdots L_{m+r-1}T^m\quad(m\ge0).
\end{equation}
\end{lemma}

\begin{proof}
The first identity follows by successive mutation through the exceptional blocks, starting with the last one. For $F\in\cA_X$ and $1\le a\le r-1$, semiorthogonality gives $\RHom_X(E_\alpha(aH),F(H))=\RHom_X(E_\alpha((a-1)H),F)=0$, so in the projection of $F(H)$ only $L_0$ acts nontrivially; this gives the second identity. Since $TL_a=L_{a+1}T$, we have $(L_0T)^m=L_0L_1\cdots L_{m-1}T^m$, and $i\Rot^m=(L_0T)^mi$ by induction. Multiplying by the first identity on the right and moving $T^m$ across the mutations gives the third.
\end{proof}

\begin{proposition}\label{prop:serre}
The rotation functor is an autoequivalence of $\cA_X$, and
\begin{equation}\label{eq:serre}
\Rot_X^{\,r}\simeq\mathsf S_{\cA_X}^{-1}[n],\qquad\mathsf S_{\cA_X}\simeq\Rot_X^{-r}[n].
\end{equation}
In particular, if $r=1$ the rotation functor is determined by the category and its dimension $n$.
\end{proposition}

\begin{proof}
Iterating the second identity of \eqref{eq:mutation} gives $i\Rot^r=L_0\cdots L_{r-1}T^ri=i\pi T^ri$, hence $\Rot^r=\pi T^ri$. For an admissible subcategory, $\mathsf S_{\cA_X}^{-1}=\pi\,\mathsf S_X^{-1}i$: for $E,F\in\cA_X$ adjunction gives
\[
\RHom_{\cA_X}(\pi\mathsf S_X^{-1}iE,F)\simeq\RHom_X(\mathsf S_X^{-1}iE,iF)\simeq\RHom_{\cA_X}(F,E)^\vee.
\]
Since $\mathsf S_X^{-1}=T^r[-n]$, the first formula follows. A functor some positive power of which is an equivalence is itself an equivalence, since $\Rot^{r-1}(\Rot^r)^{-1}$ is an inverse on both sides; this gives the second formula.
\end{proof}

\begin{remark}\label{rem:rotformula}
For $\cB_X=\langle\cO_X\rangle$ the proof of Lemma~\ref{lem:mutation} gives formula \eqref{eq:rotformula}. The right adjoint of $\Rot$ is $i^!T^{-1}i$, where $i^!$ is the right adjoint of $i$; since $\Rot$ is an equivalence, $\Rot^{-1}\simeq i^!T^{-1}i$.
\end{remark}

\subsection{Kernels and natural transformations}\label{ss:kernels}
We use the Fourier--Mukai convention $\Phi_K(F)=Rp_{2*}(p_1^*F\otimes K)$ for a perfect kernel $K$ on $X\times X$; the enhanced kernel formalism is \cite[Theorem~8.15]{Toe07} and \cite[Theorem~6.4]{LS16}. Let $Q_m$ be a kernel representing $i\Rot^m\pi$, and put
\[
\cH_m=\RHom_{\Fun(\cA_X,\cA_X)}\bigl(\id,\Rot^m\bigr).
\]
Extension $F\mapsto iF\pi$ is fully faithful on morphism complexes: restricting to $i(\cA_X)$ and composing with $\pi$ gives an inverse, and naturality with respect to the unit $\id\to i\pi$, together with $\pi i=\id$, shows that this inverse recovers a transformation on all of $\Perf(X)$. Hence
\begin{equation}\label{eq:Hm}
\cH_m\simeq\RHom_{X\times X}(Q_0,Q_m).
\end{equation}
Let $P=i\pi$, an idempotent monad on $\Perf(X)$ with unit $\eta_P\colon\id\to P$ and multiplication $\mu_P=i\varepsilon\pi\colon P\circ P\xrightarrow{\sim}P$, where $\varepsilon\colon\pi i\xrightarrow{\sim}\id$ is the counit. Expanding $\Rot^m=(\pi Ti)^m$ gives the identity of functors
\begin{equation}\label{eq:Qm}
i\Rot^m\pi=P\,T\,P\,T\cdots T\,P\qquad(m\text{ factors }T,\ m+1\text{ factors }P),
\end{equation}
and $Q_m$ is a kernel representing it. Inserting $\eta_P$ at each of the $m+1$ factors $P$ gives the \emph{coaugmentation}
\[
u_m=\eta_P*\id_T*\eta_P*\cdots*\id_T*\eta_P\colon T^m\longrightarrow i\Rot^m\pi;
\]
expanding each factor $P=L_0\cdots L_{r-1}$ by Lemma~\ref{lem:mutation}, $u_m$ is the map $\Delta_*H^{\otimes m}\to Q_m$ given by the coaugmentations of the mutations. For $s\in H^0(X,H^{\otimes m})$, multiplication by $s$ is a transformation $s\colon\id\to T^m$, and $u_m\circ s$ is a morphism of kernels $\cO_\Delta\to Q_m$. By \eqref{eq:Hm} and Lemma~\ref{lem:diagonal} below it defines an element $\theta_m(s)$, and we call
\begin{equation}\label{eq:theta}
\theta_m\colon H^0(X,H^{\otimes m})\longrightarrow H^0(\cH_m)=\Hom(\id,\Rot^m)
\end{equation}
the \emph{geometric map}. For $m=1$, $\theta_1(s)$ is the transformation $\pi(s)\colon\id=\pi i\to\pi Ti=\Rot$ induced by multiplication by $s$.

\subsection{Homological smoothness}\label{ss:smooth}
A dg category is \emph{homologically smooth} if its diagonal bimodule is perfect. Smoothness is invariant under Morita equivalence and passes to admissible components. For a projective scheme $Z$, the category $\Perf(Z)$ is homologically smooth if and only if $Z$ is smooth over $\CC$ \cite[Theorem~4.3]{LS16}; this includes nonreduced $Z$.

We also use smoothness of gluings. Let a proper enhanced category have a finite semiorthogonal decomposition with smooth and proper components. Each gluing bimodule is pointwise a perfect complex, and over the tensor product of two smooth proper categories such a bimodule is perfect, since a finite resolution of the diagonal expresses it as a finite extension of representable modules \cite[Theorem~3.18 and the proof of Theorem~4.15]{Orl16}. The glued category is then smooth by \cite[Theorem~3.25]{Orl16}, applied successively to the components.

\subsection{Projective duality}\label{ss:duality}
Let $Z\subset\PP(V^\vee)$ be an irreducible projective variety. Its projective dual $Z^\vee\subset\PP(V)$ is the closure of the set of hyperplanes containing the embedded tangent space of $Z$ at a smooth point; it is the image of the conormal variety of $Z$, which is irreducible of dimension $\dim V-2$, and it is irreducible. One expects $Z^\vee$ to be a hypersurface; otherwise $Z$ is called \emph{dual defective}, with \emph{dual defect} $\delta=\dim V-2-\dim Z^\vee$. In characteristic zero, projective biduality $(Z^\vee)^\vee=Z$ holds for every irreducible $Z$, including dual-defective ones \cite{Tev05}, \cite[Introduction and Section~1]{LM08}. If $Z$ is smooth, $Z^\vee$ is the locus of hyperplanes $\Pi$ for which the hyperplane section $Z\cap\Pi$ is singular.

\subsection{Enhanced gluing}\label{ss:gluing}
Let $\cC_X=\langle\cK_X,\cB_X\rangle$ be a semiorthogonal decomposition of a smooth proper enhanced category, with inclusions $i_X,j_X$, and put $\Gamma_X=i_X^!j_X\colon\cB_X\to\cK_X$. The mixed morphism complexes form the derived bimodule
\[
\cM_X(b,a)=\RHom_{\cC_X}(i_Xa,j_Xb)\simeq\RHom_{\cK_X}(a,\Gamma_Xb),
\]
and $\cC_X$ is recovered, up to Morita equivalence, from its two components and this bimodule by the upper-triangular construction of \cite[Section~3]{Orl16}. An identification of the gluing bimodules, including their higher actions, is therefore exactly what is needed to glue equivalences of the components.

\begin{lemma}\label{lem:gluing}
Let $\cC_X=\langle\cK_X,\cB_X\rangle$ and $\cC_Y=\langle\cK_Y,\cB_Y\rangle$ be semiorthogonal decompositions of smooth proper enhanced categories. If there are equivalences $\Phi\colon\cK_X\simeq\cK_Y$ and $\Psi\colon\cB_X\simeq\cB_Y$ with an enhanced identification $\Phi\Gamma_X\simeq\Gamma_Y\Psi$, then $\cC_X\simeq\cC_Y$.
\end{lemma}

\begin{proof}
The mixed morphisms form the derived bimodule $\cM_X(b,a)\simeq\RHom_{\cK_X}(a,\Gamma_Xb)$, naturally in both arguments and with all higher compositions, so the hypothesis identifies $\cM_X$ with $\cM_Y$ as derived bimodules along $\Phi$ and $\Psi$. The upper-triangular model then gives $\cC_X\simeq\cC_Y$.
\end{proof}

\begin{corollary}\label{cor:gluingobject}
If $\cC_X=\langle\cK_X,E_X\rangle$ with $E_X$ exceptional, then the pair $(\cK_X,i_X^!E_X)$ determines $\cC_X$. It suffices that an equivalence $\Phi\colon\cK_X\simeq\cK_Y$ carries $i_X^!E_X$ into the orbit of $i_Y^!E_Y$ under autoequivalences of $\cK_Y$ and shifts.
\end{corollary}

\begin{proof}
The subcategory generated by $E_X$ is $\Perf(\CC)$, and its gluing functor is determined by the image $i_X^!E_X$ of the generator: the gluing bimodule is the module $\RHom_{\cK_X}(-,i_X^!E_X)$, with the trivial action of $\REnd(E_X)=\CC$. In the orbit version, first compose $\Phi$ with an autoequivalence of $\cK_Y$, and send $E_X$ to the corresponding shift of $E_Y$.
\end{proof}
\section{The full kernel estimate}\label{sec:kernel}

Let $A$ be a dg algebra model of $\cA_X$ and $P$ the $A$-bimodule representing $\Rot_X^{-1}$. To a section $s\in H^0(X,H)$ we shall attach a \emph{divisor algebra}: an $A_\infty$ algebra $B=A\oplus P[1]$ which extends $A$ by $P[1]$, with connecting map $P\to A$ given by $s$ (Definition~\ref{def:ncdiv}). It is the noncommutative analogue of the structure sheaf $\cO_D=\Cone(\cO_X(-H)\xrightarrow{s}\cO_X)$ of the divisor $D$ of $s$, and its perfect modules form the residual category of $D$ (Section~\ref{sec:sections}). To show that $B$ is determined by $s$, we need control over the negative cohomology of the complexes of natural transformations $\id\to\Rot^m$ for all $m$, not only over their degree-zero cohomology. In this section we prove such a bound in the setting of \eqref{eq:setting}.

\subsection{Reduction to the diagonal}

\begin{lemma}\label{lem:diagonal}
The natural map $\cO_\Delta\to Q_0$ induces an isomorphism
\[
\RHom_{X\times X}(Q_0,Q_m)\simeq\RHom_{X\times X}(\cO_\Delta,Q_m).
\]
\end{lemma}

\begin{proof}
Expanding the $r$ mutations in $Q_0$ by \eqref{eq:mutation}, the cone of $\cO_\Delta\to Q_0$ is built from shifts of external products $E_\beta^\vee(-aH)\boxtimes E_\alpha(bH)$ with $0\le a,b\le r-1$, tensored with finite-dimensional complexes. For any kernel $K$, adjunction gives
\[
\RHom_{X\times X}\bigl(E_\beta^\vee(-aH)\boxtimes E_\alpha(bH),K\bigr)\simeq\RHom_X\bigl(E_\alpha(bH),\Phi_K(E_\beta(aH))\bigr),
\]
and for $K=Q_m$ the right-hand side vanishes because $\pi$ kills $E_\beta(aH)$.
\end{proof}

\subsection{The mutation cube}
By \eqref{eq:mutation}, $Q_m$ is obtained by expanding $m+r$ two-term mutation kernels and tensoring the input by $H^{\otimes m}$; its empty-subset term is $\Delta_*H^{\otimes m}$. A nonempty subset $J=\{a_1<\dots<a_p\}\subset\{0,\dots,m+r-1\}$ with labels $\alpha=(\alpha_1,\dots,\alpha_p)$ contributes
\begin{equation}\label{eq:cube}
C_{J,\alpha}\otimes\bigl(E^\vee_{\alpha_p}((m-a_p)H)\boxtimes E_{\alpha_1}(a_1H)\bigr)[p],\qquad C_{J,\alpha}=\bigotimes_{j=1}^{p-1}\RHom_X\bigl(E_{\alpha_j}(a_jH),E_{\alpha_{j+1}}(a_{j+1}H)\bigr).
\end{equation}
Indeed, the rightmost selected mutation evaluates against $E_{\alpha_p}(a_pH)$ and produces the input factor, each further selected mutation contributes the morphism complex between consecutive selected objects, and each selected evaluation sits in degree $-1$ of its cone. For line bundles, $C_{J,\alpha}=\bigotimes_jH^0(X,\cO_X(a_{j+1}-a_j))$ by \eqref{eq:kodaira}.

\begin{proposition}[Full kernel estimate]\label{prop:kernel}
For every $m\ge0$ there is a natural triangle
\begin{equation}\label{eq:kerneltriangle}
\RHom_{X\times X}\bigl(\cO_\Delta,\Delta_*H^{\otimes m}\bigr)\longrightarrow\cH_m\longrightarrow W_m,\qquad W_m\in D^{\ge n-r-m}(\CC).
\end{equation}
Consequently
\begin{equation}\label{eq:bound}
\cH_m\in D^{\ge b_m}(\CC),\qquad b_m=\min\{0,n-r-m\},
\end{equation}
and the geometric map $\theta_m\colon H^0(X,H^{\otimes m})\to H^0(\cH_m)$ is an isomorphism for $m\le n-r-1$ and injective for $m=n-r$.
\end{proposition}

\begin{proof}
By \eqref{eq:Hm} and Lemma~\ref{lem:diagonal}, $\cH_m\simeq\RHom(\cO_\Delta,Q_m)$. The cone of the map $\Delta_*H^{\otimes m}\to Q_m$ given by the coaugmentations of the mutations has a finite filtration with subquotients \eqref{eq:cube}. Since the normal bundle of the diagonal has determinant $\omega_X^{-1}$, duality along the diagonal sends \eqref{eq:cube} to
\begin{equation}\label{eq:dualcube}
C_{J,\alpha}\otimes\RHom_X\bigl(E_{\alpha_p},E_{\alpha_1}((m-a_p+a_1+r)H)\bigr)[p-n].
\end{equation}
All factors are morphism complexes between coherent sheaves and have no negative cohomology, so \eqref{eq:dualcube} has no cohomology below $n-p\ge n-r-m$, since $p\le m+r$. This gives \eqref{eq:kerneltriangle}. Since $\cO_\Delta$ and $\Delta_*H^{\otimes m}$ are sheaves, their derived morphisms have no negative cohomology and their degree-zero morphisms are $H^0(X,H^{\otimes m})$; the induced map to $H^0(\cH_m)$ is $\theta_m$. The remaining assertions follow from the long exact sequence
\[
H^{-1}(W_m)\longrightarrow H^0(X,H^{\otimes m})\xrightarrow{\ \theta_m\ }H^0(\cH_m)\longrightarrow H^0(W_m).
\]
For line bundles, \eqref{eq:dualcube} is $C_{J,\alpha}\otimes\RG(X,\cO_X(m-a_p+a_1+r))[p-n]$, which is even concentrated in degree $n-p$ by \eqref{eq:kodaira}, since $m-a_p+a_1+r\ge1$.
\end{proof}

\begin{corollary}\label{cor:theta1}
If $n-r\ge2$, then
\[
\theta_1\colon H^0(X,H)\xrightarrow{\ \sim\ }\Hom(\id_{\cA_X},\Rot_X),\qquad\HH^0(\cA_X)=\Hom(\id_{\cA_X},\id_{\cA_X})=\CC.
\]
If $n-r\ge1$, then $\HH^0(\cA_X)=\CC$, and for a fixed equivalence $\Phi$ any two isomorphisms $\Phi\circ\Rot_X\simeq\Rot_Y\circ\Phi$, which we call \emph{rotation intertwiners}, induce the same projective identification of the spaces $\Hom(\id,\Rot)$.
\end{corollary}

\begin{proof}
Apply Proposition~\ref{prop:kernel} with $m=1$ and $m=0$. Composition with $\Rot^{-1}$ identifies the automorphisms of $\Rot$ with the invertible elements of $\HH^0(\cA_X)=\CC$, so two rotation intertwiners differ by a scalar.
\end{proof}

\section{Noncommutative divisors}\label{sec:ncdiv}

A natural transformation $\id\to\Rot$ gives a bimodule map, but the category of a divisor depends on more: on an algebra structure, including higher products. In this section we define the relevant algebras, following Seidel \cite{Sei17}, and prove that under the kernel estimate they are determined by their sections.

\subsection{Definition}
We use cohomological grading and strictly unital $A_\infty$ algebras with normalized cochains \cite[Sections~3.1 and~3.6]{Kel01}. Let $A$ be a proper dg algebra and $P$ an invertible $A$-bimodule, and put
\begin{equation}\label{eq:Em}
E_m=\RHom_{A^e}\bigl(P^{\otimes_Am},A\bigr),\qquad A^e=A^{\mathrm{op}}\otimes A.
\end{equation}

\begin{definition}[{\cite[Definition~2.9]{Sei17}}]\label{def:ncdiv}
A \emph{noncommutative divisor} on $(A,P)$ is an $A_\infty$ structure on the graded vector space $B=A\oplus P[1]$ that restricts to the given algebra structure on $A$ and induces the given bimodule structure on the quotient $B/A=P[1]$. The inclusion and quotient give a triangle of $A$-bimodules
\begin{equation}\label{eq:divtriangle}
P\xrightarrow{\ t\ }A\longrightarrow B\longrightarrow P[1],
\end{equation}
and the class $[t]\in H^0(E_1)$ is the \emph{section} of $B$. We also call $B$ a \emph{divisor algebra}.
\end{definition}

\subsection{The weight filtration of the deformation complex}
Start from the trivial extension $A\oplus P[1]$, in which all operations with at least two inputs from $P[1]$ vanish. Give $A$ weight $0$ and $P[1]$ weight $-1$; every additional operation of a divisor has positive weight. Let $\mathfrak g_m$ be the weight-$m$ part of the normalized shifted Hochschild cochain complex, $m>0$, with differential given by the commutator with the trivial extension, and
\begin{equation}\label{eq:weight}
\mathfrak g=\prod_{m\ge1}\mathfrak g_m,\qquad F^q\mathfrak g=\prod_{m\ge q}\mathfrak g_m,\qquad[F^q\mathfrak g,F^j\mathfrak g]\subset F^{q+j}\mathfrak g.
\end{equation}
A divisor structure is $b_0+\beta$ with $\beta\in\mathfrak g^1$ satisfying $d\beta+\frac12[\beta,\beta]=0$, which is the equation $(b_0+\beta)^2=0$; relative changes of $A_\infty$ coordinates act by positive-weight gauge transformations, which preserve the subalgebra and the quotient bimodule. By \cite[(2.61)--(2.62)]{Sei17} there are exact sequences
\begin{equation}\label{eq:seidel}
H^{k-2m}(E_m)\longrightarrow H^k(\mathfrak g_m)\longrightarrow H^{k-2m+1}(E_m).
\end{equation}
For $m=1$ and $k=1$, the second map sends a first-order divisor structure to its section. To see the shifts, note that a weight-$m$ operation has either $m$ inputs from $P[1]$ and output in $A$, or $m+1$ such inputs and output in $P[1]$; in the bar construction each such input is $P[2]$, and the two resulting complexes are $\RHom_{A^e}(P^{\otimes(m+1)},P)[-2m]$ and $E_m[1-2m]$. Tensoring by $P^{-1}$ identifies the first with $E_m[-2m]$.

\subsection{Uniqueness for a fixed section}

\begin{lemma}\label{lem:uniqueness}
Assume
\begin{equation}\label{eq:vanishing}
H^{-1}(E_1)=0,\qquad H^{1-2m}(E_m)=H^{2-2m}(E_m)=0\quad(m\ge2).
\end{equation}
Then two noncommutative divisors on $(A,P)$ with the same section are isomorphic by a relative $A_\infty$ isomorphism that is the identity on $A$ and on $P[1]$.
\end{lemma}

\begin{proof}
Write the two structures as $\beta=\sum_m\beta_m$ and $\gamma=\sum_m\gamma_m$, with $\beta_m,\gamma_m\in\mathfrak g_m^1$, and argue by induction on the weight. By \eqref{eq:seidel} and the first vanishing, $H^1(\mathfrak g_1)\to H^0(E_1)$ is injective. Since the sections agree, $\beta_1-\gamma_1=d\xi_1$ for some $\xi_1\in\mathfrak g_1^0$, and conjugating by $\exp(\xi_1)$ makes the weight-one terms equal. Suppose the structures agree in all weights below $m\ge2$. The quadratic terms of the weight-$m$ Maurer--Cartan equations then coincide, so $\beta_m-\gamma_m$ is a cocycle; by the other two vanishings and \eqref{eq:seidel}, $H^1(\mathfrak g_m)=0$, so $\beta_m-\gamma_m=d\xi_m$, and conjugating by $\exp(\xi_m)$ removes the difference without changing lower weights. In each arity only finitely many of these corrections contribute, since an operation of arity $a$ has weight at most $a$, so the successive transformations compose to an $A_\infty$ isomorphism. It is the identity on $A$ and on $P[1]$ because all corrections have positive weight.
\end{proof}

\subsection{Transfer to minimal models}
A geometric construction gives a divisor algebra on a particular dg model, while an equivalence of categories identifies models only up to quasi-isomorphism. The following relative transfer places two divisor structures on a common pair; see \cite[Sections~3.3 and~8.1]{Kel01} and \cite[Main Perturbation Lemma~2.4]{Cra04}.

\begin{lemma}\label{lem:transfer}
A proper strictly unital noncommutative divisor $B=A\oplus P[1]$ is $A_\infty$ quasi-isomorphic, relative to its base and quotient, to a divisor on minimal models of $A$ and $P$, and its section is carried to the derived image of the original section. If two base algebras and quotient bimodules are identified by quasi-isomorphisms, their divisor structures can be placed on one common minimal pair.
\end{lemma}

\begin{proof}
Write the differential of $B$ as $d_0+\delta$, where $d_0$ is the sum of the differentials of $A$ and $P[1]$ and $\delta\colon P[1]\to A$. Choose contractions of $A$ and $P[1]$ onto their cohomology, unital on $A$, and take their direct sum, with inclusion $\iota$, projection $p$ and homotopy $h$. These preserve the two summands, and $\delta$ vanishes on $A$ and raises the weight, so $\delta h\delta=0$ and the perturbed differential is $p\delta\iota$. Transfer the higher operations by the tree formulas. Operations with all inputs in $A$, or with one input in $P[1]$ and output in $P[1]$, are those of the transferred algebra and bimodule; every other tree has a vertex of positive weight, so the result is again a divisor on the minimal pair. Normalized transfer preserves strict units, and the transferred triangle is the image of \eqref{eq:divtriangle}, so the section is carried to its derived image. For the comparison, a quasi-isomorphism of base algebras induces an isomorphism of minimal models, and similarly for the quotient bimodules; conjugating the positive-weight operations places both divisor structures on the common pair.
\end{proof}

\subsection{Application to the inverse rotation}
In the setting \eqref{eq:setting}, choose a classical generator $G$ of $\cA_X$, put $A=\RHom_X(G,G)$, and let $P$ represent $\Rot_X^{-1}$ under $\cA_X\simeq\Perf(A)$. Tensoring both arguments by the inverse of $P^{\otimes m}$ gives
\begin{equation}\label{eq:EmHm}
E_m\simeq\RHom_{\Fun(\cA_X,\cA_X)}\bigl(\id,\Rot_X^m\bigr)=\cH_m.
\end{equation}

\begin{proposition}\label{prop:unique}
If $n-r\ge1$, noncommutative divisors on $(A,P)$ are determined, up to relative $A_\infty$ isomorphism, by their sections.
\end{proposition}

\begin{proof}
By Proposition~\ref{prop:kernel} and \eqref{eq:EmHm}, $E_m\in D^{\ge b_m}$ with $b_m=\min\{0,n-r-m\}$. For $m=1$, $b_1=0$, so $H^{-1}(E_1)=0$. For $m\ge2$, both $1-2m$ and $2-2m$ are smaller than $0$ and than $n-r-m$, since $m+n-r>2$. Lemma~\ref{lem:uniqueness} applies.
\end{proof}

The statement does not depend on the choice of generator or of algebra model. Given an equivalence $\Phi\colon\cA_X\simeq\cA_Y$ commuting with the rotation functors, choose $G$ on one side and $\Phi(G)$ on the other; the endomorphism algebras are then quasi-isomorphic, the inverse-rotation bimodules are identified by the rotation intertwiner, and Lemma~\ref{lem:transfer} places divisor algebras on both sides on a common minimal pair, carrying their sections to the same derived image.

\section{Residual categories of hyperplane sections}\label{sec:sections}

Throughout this section we work in the setting of \eqref{eq:setting}, with a classical generator $G$ of $\cA_X$, $A=\RHom_X(G,G)$, and $P$ the bimodule of $\Rot_X^{-1}$ as in Section~\ref{sec:ncdiv}. Let $0\ne s\in H^0(X,H)$, and let $j\colon D=D_s\hookrightarrow X$ be its zero scheme, an effective Cartier divisor, possibly singular or nonreduced. We construct the divisor algebra of $D$ and show that its perfect modules form the residual category of $D$.

\subsection{The geometric divisor algebra}
Represent $G$ by a bounded complex of vector bundles. The Koszul dg algebra $K_s=[\cO_X(-1)\xrightarrow{s}\cO_X]$, in degrees $-1$ and $0$ with square-zero degree-$(-1)$ part, maps multiplicatively and quasi-isomorphically to $j_*\cO_D$. Using total \v{C}ech complexes for a finite ordered affine cover, with cup product and composition, we obtain a dg algebra
\begin{equation}\label{eq:Bs}
B_s=\check C^\bullet\bigl(\mathcal End(G)\otimes K_s\bigr)\simeq\RHom_D(Lj^*G,Lj^*G),
\end{equation}
which as a graded vector space is $A\oplus M[1]$ with $M=\RHom_X(G,G(-1))$, and which fits into the bimodule triangle
\begin{equation}\label{eq:Bstriangle}
M\xrightarrow{\ \cdot s\ }A\longrightarrow B_s;
\end{equation}
the multiplicative enhancement is \cite[Section~3]{LS16}.

\begin{lemma}\label{lem:geomdiv}
The bimodule $M$ represents $\Rot_X^{-1}$, and under $\Hom_{A^e}(P,A)\simeq\Hom(\id_{\cA_X},\Rot_X)$ the section of $B_s$ corresponds to $\theta_1(s)$. Thus $B_s$ is a noncommutative divisor on $(A,P)$ with section $\theta_1(s)$.
\end{lemma}

\begin{proof}
The right adjoint of $\Rot_X=\pi Ti$ is $i^!T^{-1}i$, and since $\Rot_X$ is an equivalence this is $\Rot_X^{-1}$ (Remark~\ref{rem:rotformula}); evaluating its bimodule on the generator gives $M$, naturally in both arguments. The transformation $\theta_1(s)$ is induced by multiplication by $s$, $\id=\pi i\to\pi Ti=\Rot_X$; its right-adjoint mate $i^!T^{-1}i\to i^!i=\id$ is multiplication by $s$ in \eqref{eq:Bstriangle}.
\end{proof}

\subsection{The residual category of a hyperplane section}
Put
\begin{equation}\label{eq:AD}
\cA_D=\bigl\langle Lj^*\cB_X,Lj^*\cB_X(H),\dots,Lj^*\cB_X((r-2)H)\bigr\rangle^{\perp}\subset\Perf(D),
\end{equation}
with $\cA_D=\Perf(D)$ when $r=1$. For $\cB_X=\langle\cO_X\rangle$ this is the category $\Ku(D,H)=\langle\cO_D,\dots,\cO_D((r-2)H)\rangle^{\perp}$ of the introduction.

\begin{lemma}\label{lem:restrict}
For $r\ge2$ the restricted blocks form a semiorthogonal sequence of completely orthogonal exceptional blocks, and
\[
\Perf(D)=\bigl\langle\cA_D,Lj^*\cB_X,\dots,Lj^*\cB_X((r-2)H)\bigr\rangle.
\]
Moreover $Lj^*(\cA_X)\subset\cA_D$.
\end{lemma}

\begin{proof}
For perfect $U,V$ on $X$, adjunction and the Koszul resolution of $j_*\cO_D$ give a triangle
\begin{equation}\label{eq:koszultriangle}
\RHom_X(U,V(-H))\longrightarrow\RHom_X(U,V)\longrightarrow\RHom_D(Lj^*U,Lj^*V).
\end{equation}
Let $0\le b\le a\le r-2$, $U=E_\alpha(aH)$ and $V=E_\beta(bH)$. The first term of \eqref{eq:koszultriangle} vanishes: if $b\ge1$, this is the semiorthogonality of $\cB_X(aH)$ and $\cB_X((b-1)H)$, since $a>b-1$; if $b=0$, Serre duality gives
\[
\RHom_X(E_\alpha(aH),E_\beta(-H))\simeq\RHom_X\bigl(E_\beta((r-1-a)H),E_\alpha\bigr)^\vee[-n],
\]
which vanishes since $1\le r-1-a\le r-1$. The second term is $\delta_{\alpha\beta}\CC$ if $a=b$ and vanishes if $a>b$. Hence the restricted objects form a semiorthogonal sequence of completely orthogonal blocks of exceptional objects. Since $D$ is proper, $\Perf(D)$ is Hom-finite, these blocks are admissible, and the decomposition follows. Finally, for $F\in\cA_X$ and $0\le a\le r-2$, both $\RHom_X(E_\alpha(aH),F(-H))=\RHom_X(E_\alpha((a+1)H),F)$ and $\RHom_X(E_\alpha(aH),F)$ vanish, so $Lj^*F\in\cA_D$ by \eqref{eq:koszultriangle}.
\end{proof}

\begin{proposition}[Generation on every hyperplane section]\label{prop:generation}
If $G$ classically generates $\cA_X$, then $Lj^*G$ classically generates $\cA_D$, even when $D$ is singular or nonreduced. Consequently $\Perf(B_s)\simeq\cA_D$.
\end{proposition}

\begin{proof}
We first show that $T_D=Lj^*G\oplus\bigoplus_{0\le a\le r-2,\,\alpha}Lj^*E_\alpha(aH)$ is a compact generator of $\Dqc(D)$. Its summands are perfect, hence compact. Let $F\in\Dqc(D)$ with $\RHom_D(T_D,F)=0$, and put $Z=Rj_*F$. By adjunction $Z$ is right orthogonal to $G$, hence to $\cA_X$. The Serre-rotated decomposition $\Perf(X)=\langle\cB_X(-rH),\dots,\cB_X(-H),\cA_X\rangle$ extends to $\Dqc(X)$, since its projections are given by perfect kernels and preserve coproducts, so $Z$ has a finite filtration with factors
\[
\bigoplus_\alpha V_{a,\alpha}\otimes E_\alpha((a-r)H),\qquad0\le a\le r-1,
\]
for complexes of vector spaces $V_{a,\alpha}$. By Serre duality, $\RHom_X(E_\beta(kH),E_\alpha((a-r)H))$ is isomorphic to $\RHom_X(E_\alpha(aH),E_\beta(kH))^\vee[-n]$, which vanishes for $a>k$ and equals $\delta_{\alpha\beta}\CC[-n]$ for $a=k$. The conditions $\RHom_X(E_\beta(kH),Z)=0$ for $k=0,1,\dots,r-2$, which follow from the orthogonality of $F$ to the restricted blocks, therefore kill the factors one by one, and we are left with
\begin{equation}\label{eq:leftover}
Z\simeq\bigoplus_\alpha V_\alpha\otimes E_\alpha(-H);
\end{equation}
for $r=1$ this is immediate. Now $Z$ is supported on the proper divisor $D$, so it vanishes at the generic point $\eta$ of $X$. Each $(E_\alpha)_\eta$ is a nonzero vector space over $\CC(X)$, because $E_\alpha$ has positive rank, and extension of scalars is faithful, so all $V_\alpha$ vanish. Hence $Z=0$, and $F=0$ since pushforward along a closed immersion is conservative.

Thus $\Perf(D)=\thick(T_D)$. The left projection to $\cA_D$ kills the exceptional summands and fixes $Lj^*G$ by Lemma~\ref{lem:restrict}, so $\cA_D=\thick(Lj^*G)$, and \eqref{eq:Bs} with Morita theory gives $\Perf(B_s)\simeq\cA_D$.
\end{proof}

\begin{theorem}[Recovery of hyperplane sections]\label{thm:recovery}
Assume $n-r\ge2$. Then the Morita class of $\cA_{D_s}$ is determined by $(\cA_X,\Rot_X)$ and $[s]\in\PP\Hom(\id_{\cA_X},\Rot_X)=\PP H^0(X,H)$: it is the category of perfect modules over the unique noncommutative divisor on $(A,P)$ with section $\theta_1(s)$. An equivalence of such categories commuting with the rotation functors identifies the residual categories of corresponding hyperplane sections.
\end{theorem}

\begin{proof}
Corollary~\ref{cor:theta1} recovers the space of sections. By Lemma~\ref{lem:geomdiv}, $B_s$ is a noncommutative divisor on $(A,P)$ with section $\theta_1(s)$, and by Proposition~\ref{prop:unique} it is determined by this section. Under an equivalence, the comparison following Proposition~\ref{prop:unique} places the two divisor algebras on a common pair, and Proposition~\ref{prop:generation} identifies their perfect modules with the residual categories. Rescaling $s$ is absorbed by rescaling $P[1]$, so only $[s]$ matters.
\end{proof}

The construction works over any field extension of $\CC$, which will be used in Section~\ref{sec:finite} for divisors defined over function fields. For a field extension $F/\CC$ we write $X_F$, $\cA_{X,F}$ and so on for the scalar extensions.

\begin{corollary}\label{cor:extension}
Let $F/\CC$ be a field extension. The recovery of the residual categories of divisors in Theorem~\ref{thm:recovery}, and in Theorem~\ref{thm:general} below when $n-r=1$ and $\theta_1$ is an isomorphism, holds for every nonzero $s\in H^0(X,H)\otimes_\CC F$. In particular, an equivalence $\cA_X\simeq\cA_Y$ over $\CC$ commuting with the rotation functors induces, for corresponding sections over $F$, an $F$-linear equivalence of the residual categories of the corresponding divisors on $X_F$ and $Y_F$.
\end{corollary}

\begin{proof}
The projection kernels and the inverse-rotation bimodules are perfect, so their morphism complexes and their finite mutation filtrations commute with the flat extension of scalars; hence the bounds of Proposition~\ref{prop:kernel} and the identification of sections hold over $F$. A compact generator remains a compact generator after extension of scalars, since the projection $X_F\to X$ is affine with conservative pushforward, so that vanishing against the pullback is detected by adjunction; the same holds for the residual component, through its projection kernel. For a section over $F$ we use the Koszul algebra \eqref{eq:Bs} over $F$. The proofs of Lemma~\ref{lem:restrict} and Proposition~\ref{prop:generation} only use this algebra, the semiorthogonal decomposition and the positivity of the generic ranks of the exceptional sheaves, and they apply to the Cartier divisor of the section on the integral variety $X_F$. Relative transfer (Lemma~\ref{lem:transfer}) and the weight-by-weight uniqueness of Lemma~\ref{lem:uniqueness} hold over every field of characteristic zero. Hence the two geometric divisor algebras, placed on a common pair after extending the bimodule of the equivalence, have the same section and are $F$-linearly quasi-isomorphic.
\end{proof}

\section{The categorical discriminant and the main theorems}\label{sec:main}

We combine Sections~\ref{sec:kernel}--\ref{sec:sections} to prove the general reconstruction theorem (Theorem~\ref{thm:general}), and deduce Theorem~\ref{thm:sections}, Theorem~\ref{thm:A} for very ample polarizations, and Corollary~\ref{cor:vaPerry}.

\subsection{Smoothness detects singular divisors}

\begin{lemma}\label{lem:smoothness}
For every divisor $D\in|H|$, the category $\cA_D$ is homologically smooth if and only if $D$ is smooth.
\end{lemma}

\begin{proof}
If $D$ is smooth, so is $\Perf(D)$ and its admissible component $\cA_D$. Conversely, if $\cA_D$ is smooth, every component of the decomposition of Lemma~\ref{lem:restrict} is smooth and proper, and $\Perf(D)$ is proper, so the gluing criterion of Section~\ref{ss:smooth}, applied successively, shows that $\Perf(D)$ is smooth; hence $D$ is smooth.
\end{proof}

Put $V_X=\Hom(\id_{\cA_X},\Rot_X)$. For $0\ne\zeta\in V_X$ for which a noncommutative divisor on $(A,P)$ with section $\zeta$ exists, this divisor is unique when $n-r\ge1$ (Proposition~\ref{prop:unique}); we denote it by $B_\zeta$. Define the \emph{categorical discriminant}
\begin{equation}\label{eq:disc}
\Delta(\cA_X,\Rot_X)=\bigl\{[\zeta]\in\PP(V_X):\ B_\zeta\ \text{exists and}\ \Perf(B_\zeta)\ \text{is not homologically smooth}\bigr\},
\end{equation}
with its reduced structure. If $\theta_1$ is an isomorphism, then $B_{\theta_1(s)}=B_s$ exists for every $s$, and by Proposition~\ref{prop:generation} and Lemma~\ref{lem:smoothness}, under $V_X\simeq H^0(X,H)$ the discriminant is the locus of singular divisors in $|H|$. This locus is closed: it is the image of the incidence in $X\times\PP H^0(X,H)$ where the first jet of the universal section vanishes. If $H$ is very ample, a hyperplane section of $X\subset\PP(H^0(X,H)^\vee)$ is singular at $x$ exactly when the hyperplane contains the embedded tangent space $T_xX$, so
\begin{equation}\label{eq:discdual}
\Delta(\cA_X,\Rot_X)=X^\vee.
\end{equation}

\subsection{The general theorem}

\begin{theorem}\label{thm:general}
Let $X$ be as in \eqref{eq:setting}, and assume either that $n-r\ge2$, or that $n-r=1$ and $\theta_1\colon H^0(X,H)\to\Hom(\id_{\cA_X},\Rot_X)$ is an isomorphism. Then the pair $(\cA_X,\Rot_X)$ determines, compatibly with equivalences commuting with the rotation functors:
\begin{enumerate}[label=\textup{(\arabic*)}]
\item the projective space $|H|=\PP H^0(X,H)$;
\item for every $[s]\in|H|$, the Morita class of the residual category $\cA_{D_s}$;
\item the reduced locus of singular divisors in $|H|$.
\end{enumerate}
Precisely, let $Y$, with $H_Y$, $r_Y$ and $\cB_Y$, satisfy the same assumptions, with possibly different numerical data, and let $\Phi\colon\cA_X\simeq\cA_Y$ commute with the rotation functors. Then $\Phi$ induces a linear isomorphism $\alpha\colon H^0(X,H)\simeq H^0(Y,H_Y)$, unique up to scalar, such that $\cA_{D_s}\simeq\cA_{D_{\alpha(s)}}$ for every $s\ne0$, and $\PP(\alpha)$ maps the locus of singular divisors in $|H|$ onto that in $|H_Y|$. If moreover $H$ and $H_Y$ are very ample, there is an isomorphism $f\colon X\xrightarrow{\sim}Y$ with $f^*H_Y\simeq H$, induced by the projective transformation dual to $\PP(\alpha)$; in particular $\dim X=\dim Y$ and $r=r_Y$.
\end{theorem}

\begin{proof}
If $n-r\ge2$, Corollary~\ref{cor:theta1} gives (1); otherwise it is the assumption. Let $\alpha=(\theta_1^Y)^{-1}\circ\Phi_*\circ\theta_1^X$, where $\Phi_*\colon\Hom(\id,\Rot_X)\to\Hom(\id,\Rot_Y)$ is conjugation by $\Phi$ followed by the rotation intertwiner. By Corollary~\ref{cor:theta1}, which only needs $n-r\ge1$, the projective class of $\alpha$ does not depend on the choice of intertwiner.

For (2), Proposition~\ref{prop:unique} applies on both sides. By Lemma~\ref{lem:geomdiv} and the comparison following Proposition~\ref{prop:unique}, the divisor $B_s$ on $(A_X,P_X)$ is transported to a noncommutative divisor on a minimal model of $(A_Y,P_Y)$ with section $\theta_1^Y(\alpha(s))$. The geometric divisor $B_{\alpha(s)}$ on $(A_Y,P_Y)$ has the same section, so the two are relatively $A_\infty$ isomorphic, and Proposition~\ref{prop:generation} gives $\cA_{D_s}\simeq\Perf(B_s)\simeq\Perf(B_{\alpha(s)})\simeq\cA_{D_{\alpha(s)}}$. Part (3) follows from (2) and Lemma~\ref{lem:smoothness}.

Suppose now that $H$ and $H_Y$ are very ample. Then $X\subset\PP(H^0(X,H)^\vee)$ and $Y\subset\PP(H^0(Y,H_Y)^\vee)$ are embedded by complete linear systems, and by (3) and \eqref{eq:discdual}, $\PP(\alpha)(X^\vee)=Y^\vee$. The contragredient projective transformation $\PP((\alpha^\vee)^{-1})$ therefore maps $(X^\vee)^\vee$ onto $(Y^\vee)^\vee$, that is, $X$ onto $Y$ by projective biduality (Section~\ref{ss:duality}). Since both varieties are reduced, it restricts to an isomorphism of schemes $f$, which identifies the restricted hyperplane bundles.
\end{proof}

\subsection{Proof of Theorem~\ref{thm:sections}, and of Theorem~\ref{thm:A} for very ample polarizations}

\begin{proof}[Proof of Theorem~\ref{thm:sections}]
Take $\cB_X=\langle\cO_X\rangle$, so that $\cA_X=\Ku(X,H)$ and $\cA_D=\Ku(D,H)$, and apply Theorem~\ref{thm:general}, which is available since $n-r\ge2$. This gives (1) and (3), and Theorem~\ref{thm:recovery} gives the description of $\Ku(D,H)$ in (2); Corollary~\ref{cor:extension} gives the same over every field extension of $\CC$. If $H$ is very ample, \eqref{eq:discdual} identifies the locus in (3) with $X^\vee$. For $H=-K_X$ we have $r=1$, $\Ku(D,H)=\Perf(D)$ and $n-r=n-1\ge2$; the rotation functor is $\mathsf S^{-1}[n]$ by Proposition~\ref{prop:serre}, so it commutes with every equivalence between such categories for varieties of the same dimension.
\end{proof}

\begin{proof}[Proof of Theorem~\ref{thm:A} when $H_X$ and $H_Y$ are very ample]
Apply Theorem~\ref{thm:general} with $\cB_X=\langle\cO_X\rangle$ and $\cB_Y=\langle\cO_Y\rangle$, so that $\cA_X=\Ku(X,H_X)$ and $\cA_Y=\Ku(Y,H_Y)$. The assumption $r\le n-2$ is $n-r\ge2$ on both sides, and $H_X,H_Y$ are very ample, so Theorem~\ref{thm:general} gives an isomorphism $(X,H_X)\simeq(Y,H_Y)$, identifying the embeddings by the complete linear systems.
\end{proof}

\subsection{Perry's conjecture}
For $(H,r)=(-K_X,1)$ the rotation functor of $\cA_X$ is $\mathsf S^{-1}_{\cA_X}[\dim X]$ by Proposition~\ref{prop:serre}. To deduce Theorem~\ref{thm:B} and Corollary~\ref{cor:vaPerry} from the results for rotation functors, it remains to see that the dimension is detected by the category.

\begin{lemma}\label{lem:dim}
Let $X$ and $Y$ be smooth Fano varieties of dimensions $n,m\ge2$ with $H^0(X,-K_X)\ne0$ and $H^0(Y,-K_Y)\ne0$, and let $\cB_X$, $\cB_Y$ be completely orthogonal collections of exceptional coherent sheaves of positive rank, with orthogonals $\cA_X=\cB_X^{\perp}$, $\cA_Y=\cB_Y^{\perp}$. If $\cA_X\simeq\cA_Y$, then $n=m$, and every equivalence $\cA_X\simeq\cA_Y$ commutes with the rotation functors for the anticanonical polarizations.
\end{lemma}

\begin{proof}
Let $\Phi\colon\cA_X\simeq\cA_Y$. It commutes with the Serre functors, so $\Phi\Rot_X\Phi^{-1}\simeq\mathsf S_{\cA_Y}^{-1}[n]=\Rot_Y[n-m]$, and therefore
\[
\Hom(\id,\Rot_Y)\simeq\Hom^{m-n}(\id,\Rot_X)=H^{m-n}(\cH_1^X).
\]
The left-hand side contains $\theta_1^Y(H^0(Y,-K_Y))\ne0$, since $\theta_1^Y$ is injective by Proposition~\ref{prop:kernel} ($m-1\ge1$). The right-hand side vanishes if $m<n$, since $\cH_1^X\in D^{\ge0}$ by \eqref{eq:bound} ($n-1\ge1$). Hence $m\ge n$, and by symmetry $m=n$. Then $\Phi\Rot_X\Phi^{-1}\simeq\Rot_Y$.
\end{proof}

\begin{proof}[Proof of Corollary~\ref{cor:vaPerry}]
Let $X$ and $Y$ be smooth Fano varieties of dimension at least three with very ample anticanonical bundles, and let $\cB_X,\cB_Y$ be either $\langle\cO\rangle$ or completely orthogonal collections of positive-rank exceptional sheaves, with an equivalence $\cB_X^{\perp}\simeq\cB_Y^{\perp}$. By Lemma~\ref{lem:dim}, $\dim X=\dim Y=n$ and the equivalence commutes with the rotation functors $\mathsf S^{-1}[n]$. Theorem~\ref{thm:general} with $(H,r)=(-K,1)$, which satisfies $n-r=n-1\ge2$, gives an isomorphism $X\simeq Y$ identifying the anticanonical embeddings.
\end{proof}

\begin{remark}\label{rem:dP5}
For pairs of del Pezzo surfaces of degree at least five, Perry's conjecture is elementary. For a del Pezzo surface $S$ of degree $d$, the Grothendieck group $K_0(S)$ is free of rank $12-d$, so $K_0(\cO_S^{\perp})$ has rank $11-d$, and an equivalence $\cO_S^{\perp}\simeq\cO_{S'}^{\perp}$ preserves the degree. In degrees $5,6,7,9$ there is a unique del Pezzo surface. In degree $8$ there are two, $\PP^1\times\PP^1$ and the blow-up $F_1$ of $\PP^2$ in a point. For $\PP^1\times\PP^1$, $\cO^{\perp}=\langle\cO(-1,-1),\cO(0,-1),\cO(-1,0)\rangle$ and the Gram matrix of the Euler form is $\begin{psmallmatrix}1&2&2\\0&1&0\\0&0&1\end{psmallmatrix}$; for $F_1$, with $h$ the pullback of a line and $E$ the exceptional curve, Orlov's blow-up formula \cite{Orl92} gives $\cO^{\perp}=\langle\cO_E(-1),\cO(-2h),\cO(-h)\rangle$ with Gram matrix $\begin{psmallmatrix}1&-1&-1\\0&1&3\\0&0&1\end{psmallmatrix}$. The symmetrized Euler form takes only even values in the first case but not in the second, and the Euler form on $K_0$ is preserved by equivalences, so the two categories are not equivalent.
\end{remark}

\subsection{Variants}
In Theorem~\ref{thm:general}, very ampleness enters only through \eqref{eq:discdual}. The following variant replaces it by a weaker condition, using normalization alone.

\begin{proposition}\label{prop:unramified}
In Theorem~\ref{thm:general}, and hence in Theorem~\ref{thm:A}, very ampleness of $H$ and $H_Y$ may be replaced by the condition that the linear systems are base point free and define unramified morphisms.
\end{proposition}

\begin{proof}
Let $\varphi\colon X\to\PP(V^\vee)$, $V=H^0(X,H)$, be the morphism defined by $|H|$, assumed unramified, and let $Z=\varphi(X)$. Since $H$ is ample, $\varphi$ is finite. It factors through the normalization $\nu\colon Z^\nu\to Z$ as $\psi\colon X\to Z^\nu$, which is finite, surjective and unramified, since $\Omega_{X/Z^\nu}$ is a quotient of $\Omega_{X/\PP(V^\vee)}=0$. By the local structure of unramified morphisms \cite[18.4.7]{EGAIV4}, étale locally $X$ is a closed subscheme of a scheme étale over $Z^\nu$, hence normal, of the same dimension; as $X$ is integral, the two agree, so $\psi$ is étale. Then $Z^\nu$ is smooth, and $-K_{Z^\nu}$ is ample because $\psi^*(-K_{Z^\nu})=-K_X$; thus $Z^\nu$ is a smooth Fano variety, hence simply connected, and the connected étale cover $\psi$ is an isomorphism. So $\varphi$ is the normalization of $Z$.

A divisor $D_s=\varphi^{-1}(\Pi_s)$ is singular at $x$ exactly when the hyperplane $\Pi_s$ contains the embedded tangent space $\varphi(x)+d\varphi_x(T_xX)$, which is $n$-dimensional since $\varphi$ is unramified. The incidence of such pairs $(x,[s])$ is a projective bundle over $X$, hence irreducible, and over the dense open set where $\varphi$ is an isomorphism onto the smooth locus of $Z$ it is the conormal variety of $Z$. Hence the locus of singular divisors is $Z^\vee$, biduality recovers $Z\subset\PP(V^\vee)$, and $X$ is recovered as its normalization, with $H=\nu^*\cO(1)$. The argument of Theorem~\ref{thm:general} then applies verbatim.
\end{proof}

\begin{remark}\label{rem:limits}
If $|H|$ defines a finite morphism with ramification, parts (1)--(3) of Theorem~\ref{thm:general} still hold, and the locus of singular divisors also records the ramification. For instance, let $f\colon X\to M\subset\PP^N$ be a double cover branched along a smooth divisor $B$, with $H=f^*\cO(1)$. Away from $B$ the map $f$ is étale. Near a point of $B$, in local coordinates with $B=\{u_1=0\}$ and $X$ given by $y^2=u_1$, the pullback $h(y^2,u_2,\dots,u_n)$ of a local equation $h$ of $M\cap\Pi$ has vanishing $y$-derivative along $y=0$, and its other derivatives are those of $h|_B$. Hence the singular divisors in $|H|$ are those tangent to $M$ or to $B$, and the discriminant is $M^\vee\cup B^\vee$. Section~\ref{sec:finite} recovers the image from the component $M^\vee$, when it is a hypersurface, and the covering from a tangent divisor over the function field of $M^\vee$.
\end{remark}

\section{Finite polarization maps and tangent divisors}\label{sec:finite}

When $|H|$ defines a finite morphism which is not an embedding, the locus of singular divisors is no longer the dual of the image: the ramification contributes further components (Remark~\ref{rem:limits}). In this section we first recover the image from a hypersurface component of this locus, and then the covering from a single divisor defined over the function field of the dual of the image. Throughout, equivalences are enhanced and linear over the field indicated.

\subsection{Incidence and the common image}
We use the following hypersurface case of generalized biduality; the inclusion of dual varieties is \cite[Theorem~(1.4)]{LM08}, and the differential argument below also locates every singular point of a general member.

\begin{lemma}\label{lem:incidence}
Let $W$ be a smooth projective complex variety, $L$ a globally generated line bundle, and $V\subset H^0(W,L)$ a generating subspace, with induced morphism $f\colon W\to\PP(V^\vee)$. Let $\Gamma\subset\PP(V)$ be an irreducible hypersurface which is a component of the locus of singular divisors, and let $\gamma_\Gamma\colon\Gamma\dashrightarrow\PP(V^\vee)$ be its Gauss map. Then for a general $[s]\in\Gamma$ every singular point $w$ of $D_s$ satisfies $f(w)=\gamma_\Gamma([s])$. In particular $\Gamma^\vee\subset f(W)$.
\end{lemma}

\begin{proof}
Consider the reduced closed incidence $I=\{(w,[s]):s(w)=0,\ ds_w=0\}\subset W\times\PP(V)$, and let $T$ be an irreducible component of $I$ dominating $\Gamma$. In characteristic zero the map $T\to\Gamma$ is smooth at a general point $(w,[s])$. Differentiating $s(w)=0$ along $T$ gives $\dot s(w)+ds_w(\dot w)=0$, hence $\dot s(w)=0$, since $ds_w=0$. Thus $T_{[s]}\Gamma$ is contained in the kernel of evaluation at $w$ on $T_{[s]}\PP(V)$. This evaluation is nonzero because $V$ generates $L$, and both subspaces are hyperplanes, so they are equal; the point of $\PP(V^\vee)$ dual to the first is $\gamma_\Gamma([s])$, and that dual to the second is $f(w)$.

On the inverse image of the smooth locus of $\Gamma$, the equality $f(w)=\gamma_\Gamma([s])$ of morphisms extends from a dense open subset of $T$ to all of it, since projective space is separated. As $I$ has finitely many components, removing from $\Gamma$ the closures of the images of the components which do not dominate it, and intersecting the open sets obtained for the dominating ones, proves the assertion for every singular point of a general member. Finally, the closure of the image of the Gauss map of $\Gamma$ is $\Gamma^\vee$.
\end{proof}

\begin{lemma}\label{lem:commonimage}
Let $X,Y$ be smooth projective complex varieties of the same dimension, with ample globally generated line bundles $H_X,H_Y$, and assume that $Z_{H_X}^\vee$ is a hypersurface. Then every linear isomorphism $H^0(X,H_X)\simeq H^0(Y,H_Y)$ which identifies the reduced loci of singular divisors identifies the images $Z_{H_X}$ and $Z_{H_Y}$.
\end{lemma}

\begin{proof}
Use the identification of section spaces to place both images in $\PP(V^\vee)$. The maps $f_{H_X}$ and $f_{H_Y}$ are finite, since a positive-dimensional fibre would contain a curve on which an ample line bundle has degree zero; so the images have the same dimension. Put $\Gamma=Z_{H_X}^\vee$. Over a dense smooth open subset of $Z_{H_X}$ the map from $X$ is étale, so a hyperplane tangent there pulls back to a singular divisor, and $\Gamma$ lies in the locus of singular divisors of $X$. By Bertini's theorem this locus is a proper closed subset, so the irreducible hypersurface $\Gamma$ is one of its components, and hence also a component of the locus of singular divisors of $Y$. Lemma~\ref{lem:incidence} for $Y$ and biduality give $Z_{H_X}=\Gamma^\vee\subset Z_{H_Y}$. The two images are irreducible and reduced of the same dimension, hence equal.
\end{proof}

\subsection{The singularity category of a residual component}
For a small proper idempotent-complete dg category $\cT$ over a field $F$, let $\PsPerf(\cT)$ be the dg category of $\cT$-modules with values in perfect $F$-complexes. It contains the Yoneda image of $\cT$, and we put
\[
\Sg(\cT)=\bigl(\PsPerf(\cT)/\cT\bigr)^{\natural},
\]
where the quotient is the enhanced Verdier quotient and $(-)^\natural$ is idempotent completion. This construction is Morita invariant. For a projective $F$-scheme $D$, \cite[Proposition~B.14]{Efi17} identifies $\PsPerf(\Perf(D))$ with $\Db_{\mathrm{coh}}(D)$, so $\Sg(\Perf(D))$ is the idempotent completion of the singularity category of $D$.

\begin{lemma}\label{lem:sgresidual}
Let $D$ be a projective $F$-scheme with a semiorthogonal decomposition $\Perf(D)=\langle\cT,E_1,\dots,E_t\rangle$, where $E_1,\dots,E_t$ is an exceptional collection. Then there is an enhanced equivalence
\[
\Sg(\cT)\simeq\bigl(\Db_{\mathrm{coh}}(D)/\Perf(D)\bigr)^{\natural},
\]
compatible with automorphisms of $F$ acting on $D$ and on the decomposition. In particular, the residual category $\cA_D$ of \eqref{eq:AD} determines the singularity category of $D$.
\end{lemma}

\begin{proof}
Let $i\colon\cT\to\Perf(D)$ be the inclusion and $\pi$ its left adjoint. Precomposition with $\pi$ extends a pseudo-perfect module on $\cT$ to one on $\Perf(D)$; this is fully faithful, preserves representable modules, and its image consists of the modules vanishing on every $E_j$, since the unit triangle of the projection has its third term in $\langle E_1,\dots,E_t\rangle$. Under Efimov's identification this image is the right orthogonal of the $E_j$ in $\Db_{\mathrm{coh}}(D)$. For a coherent complex $M$ all $\RHom_D(E_j,M)$ are perfect over $F$, because the $E_j$ are perfect and $D$ is proper, so the usual mutation triangles extend the decomposition to $\Db_{\mathrm{coh}}(D)=\langle\PsPerf(\cT),E_1,\dots,E_t\rangle$, with perfect part $\langle\cT,E_1,\dots,E_t\rangle$. Taking quotients kills the exceptional summands on both sides and gives the equivalence after idempotent completion. All constructions are natural, which gives the compatibility with automorphisms.
\end{proof}

\begin{lemma}\label{lem:nodes}
Let $D,D'$ be projective schemes of positive dimension over a field $F$ of characteristic zero, all of whose geometric singularities are isolated ordinary double points, and suppose that $\Perf(D)$ and $\Perf(D')$ have semiorthogonal decompositions into residual components $\cT,\cT'$ and finite, possibly empty, exceptional collections. Then an $F$-linear equivalence $\cT\simeq\cT'$ induces an isomorphism of finite étale $F$-schemes $\Sing(D)_{\mathrm{red}}\simeq\Sing(D')_{\mathrm{red}}$.
\end{lemma}

\begin{proof}
Choose a finite Galois extension $E/F$ over which the singular points and their quadratic tangent cones are split. Flat base change of compact generators identifies the residual components over $E$ with the scalar extensions of $\cT$ and $\cT'$, so the equivalence gives an $E$-linear equivalence compatible with the semilinear Galois descent data. Apply Lemma~\ref{lem:sgresidual} over $E$. By \cite[Proposition~2.7, Corollary~2.9 and Theorem~2.10]{Orl11}, the idempotent completion of the singularity category is a product of local factors, one for each singular point, and this decomposition is natural with respect to permutations of the supports. At a split ordinary double point the formal Morse lemma reduces the local equation to a nondegenerate split quadratic form, and the local factor is described by the corresponding Clifford superalgebra \cite[Section~2, Theorem~5.7 and Section~5.5]{Dyc11}; the split Clifford superalgebra is graded Morita equivalent to $E$ or to the Clifford superalgebra on one generator. Hence each local factor has a single orbit of indecomposable objects under the suspension, and it is not a nontrivial product of triangulated categories. An exact equivalence therefore permutes the local factors, and the permutation is equivariant for the Galois group of $E/F$, because the equivalence and the support decomposition are compatible with the descent data. Since finite étale $F$-schemes are equivalent to finite continuous Galois sets \cite[Tag~03QR]{Stacks}, the equivariant bijection descends to the stated isomorphism.
\end{proof}

\subsection{A generic tangent divisor}

\begin{lemma}\label{lem:tangent}
Let $f\colon W\to Z\subset\PP^N$ be a finite surjective morphism from a smooth complex variety of dimension $n\ge2$, and assume that $\Gamma=Z^\vee$ is a hypersurface. Put $K=\CC(Z)$ and $F=\CC(\Gamma)$, and let $D_\eta$ be the pullback to $W_F$ of the \emph{generic tangent hyperplane}, the hyperplane over $F$ given by the generic point of $\Gamma$. The \emph{contact map} $\Gamma\dashrightarrow Z$, sending a general tangent hyperplane to its point of tangency, embeds $K$ in $F$, and
\begin{equation}\label{eq:puretrans}
F\simeq K(t_1,\dots,t_{N-n-1}).
\end{equation}
Every geometric singularity of $D_\eta$ is an ordinary double point, and the reduced singular scheme of $D_\eta$ is
\begin{equation}\label{eq:nodealgebra}
\Sing(D_\eta)_{\mathrm{red}}\simeq\Spec\bigl(\CC(W)\otimes_KF\bigr).
\end{equation}
\end{lemma}

\begin{proof}
Over the smooth locus of $Z$ the conormal variety is the projective bundle of tangent hyperplanes, with fibres $\PP^{N-n-1}$. By reflexivity its second projection is the conormal correspondence of $\Gamma$, and since $\Gamma$ is a hypersurface, its general point has a unique tangent hyperplane, so this projection is birational. Taking function fields gives \eqref{eq:puretrans}.

At a general conormal pair $(z,\Pi)$ the projection to $\Gamma$ is a local isomorphism. In local coordinates the contact equations are $h(z)=0$ and $dh_z=0$, where $h$ is a local equation of the hyperplane section, and with $\Pi$ fixed the kernel of the differential of the conormal projection is the kernel of $\mathrm{Hess}_z(h)$. Its vanishing shows that this Hessian is nondegenerate, so the generic tangent hyperplane section has an ordinary double point at its contact point. The contact point maps to the generic point of $Z$, which lies in the open set over which $f$ is étale, so the fibre over it consists of ordinary double points of $D_\eta$. There are no other singularities: by Bertini's theorem a general member of the pulled-back linear system is smooth, so $\Gamma$ is a hypersurface component of the locus of singular divisors, and Lemma~\ref{lem:incidence} confines every singular point over its generic point to the contact fibre. This fibre is the generic fibre $\Spec\CC(W)$ of $f$, base changed along $K\subset F$, which gives \eqref{eq:nodealgebra}.
\end{proof}

\begin{lemma}\label{lem:fielddescent}
Let $K$ be a field of characteristic zero, $F=K(t_1,\dots,t_a)$, and $E_1,E_2$ finite field extensions of $K$. If $E_1\otimes_KF\simeq E_2\otimes_KF$ as $F$-algebras, then $E_1\simeq E_2$ over $K$.
\end{lemma}

\begin{proof}
A finite extension is linearly disjoint from a purely transcendental one, so $E_i\otimes_KF$ is the field $E_i(t_1,\dots,t_a)$, and $E_i$ is relatively algebraically closed in it; hence $E_i$ is the set of its elements algebraic over $K$. An $F$-algebra isomorphism fixes $K$, and therefore maps $E_1$ onto $E_2$.
\end{proof}

\subsection{Reconstruction of the finite map}

\begin{theorem}\label{thm:finite}
Let $X$ and $Y$ be as in \eqref{eq:setting}, of the same dimension $n\ge3$, with possibly different blocks $\cB_X,\cB_Y$ and integers $r_X,r_Y$. On each side assume either that $n-r\ge2$, or that $n-r=1$ and $\theta_1$ is an isomorphism. Assume that $H_X$ and $H_Y$ are globally generated and that $Z_{H_X}^\vee$ is a hypersurface. Then every equivalence $\cA_X\simeq\cA_Y$ commuting with the rotation functors induces a projective identification of $Z_{H_X}$ and $Z_{H_Y}$ which lifts to an isomorphism $g\colon X\xrightarrow{\sim}Y$ with $g^*H_Y\simeq H_X$. In particular $r_X=r_Y$.
\end{theorem}

\begin{proof}
Theorem~\ref{thm:general} identifies the section spaces and the reduced loci of singular divisors, and Lemma~\ref{lem:commonimage} identifies the images; let $Z$ be the common image. Put $\Gamma=Z^\vee$, $K=\CC(Z)$ and $F=\CC(\Gamma)$, and choose an $F$-rational representative of the generic point of $\Gamma$. By Corollary~\ref{cor:extension}, the equivalence induces an $F$-linear equivalence of the residual categories of the two generic tangent divisors on $X_F$ and $Y_F$. By Lemma~\ref{lem:tangent} these divisors have only geometric ordinary double points; their exceptional decompositions are those of Lemma~\ref{lem:restrict}, or the empty one for $r=1$. Lemma~\ref{lem:nodes} identifies their reduced singular schemes, and \eqref{eq:nodealgebra} gives an isomorphism of $F$-algebras
\[
\CC(X)\otimes_KF\simeq\CC(Y)\otimes_KF.
\]
By \eqref{eq:puretrans} and Lemma~\ref{lem:fielddescent}, $\CC(X)\simeq\CC(Y)$ over $K$. A normal variety finite over $Z$ is the normalization of $Z$ in its function field, even if $Z$ is not normal: on an affine open subset $\Spec R\subset Z$ its coordinate ring $B$ is finite over $R$ and integrally closed in its fraction field, and every element integral over $R$ is integral over $B$, so $B$ is the integral closure of $R$. Hence the isomorphism of function fields lifts to an isomorphism $g\colon X\simeq Y$ over $Z$. Both polarizations are pullbacks of $\cO_Z(1)$, so $g^*H_Y\simeq H_X$, and $-K_X=r_XH_X=r_YH_X$ with $H_X$ ample forces $r_X=r_Y$.
\end{proof}

\begin{lemma}\label{lem:vadual}
Let $X$ be a smooth complex Fano variety with very ample anticanonical bundle. Then the projective dual of its anticanonical embedding is a hypersurface.
\end{lemma}

\begin{proof}
The anticanonical image is not a linearly embedded projective space, since $-K_{\PP^n}\ne\cO_{\PP^n}(1)$. Suppose that its dual defect $\delta$ is positive. By Ein's theorem on contact loci \cite{Ein86}, in the form of \cite[Proposition~3.1]{IR14}, a general contact locus is a linear space $\PP^\delta$, and a line $\ell$ in it has $\deg N_{\ell/X}=(n+\delta-2)/2$. Hence $-K_X\cdot\ell=2+\deg N_{\ell/X}=(n+\delta+2)/2>1$. But a line in the anticanonical embedding has $-K_X\cdot\ell=1$, a contradiction.
\end{proof}

\subsection{Proofs of Theorems~\ref{thm:A} and~\ref{thm:B}}

\begin{proof}[Proof of Theorem~\ref{thm:A}]
If $H_X$ and $H_Y$ are very ample, this was proved in Section~\ref{sec:main}. Otherwise, apply Theorem~\ref{thm:finite} with $\cB_X=\langle\cO_X\rangle$ and $\cB_Y=\langle\cO_Y\rangle$, so that $\cA_X=\Ku(X,H_X)$ and $\cA_Y=\Ku(Y,H_Y)$; the assumption $r\le n-2$ is $n-r\ge2$ on both sides.
\end{proof}

\begin{proof}[Proof of Theorem~\ref{thm:B}]
Globally generated ample line bundles have nonzero sections, so Lemma~\ref{lem:dim} gives a common dimension $n\ge3$ and shows that the equivalence commutes with the anticanonical rotation functors. Theorem~\ref{thm:finite} with $(H,r)=(-K,1)$, so that $n-r\ge2$, applies to $\langle\cO\rangle$ and to every completely orthogonal collection of positive-rank exceptional sheaves, and gives an isomorphism $X\simeq Y$ over the common anticanonical image.
\end{proof}

By Lemma~\ref{lem:vadual}, Corollary~\ref{cor:vaPerry} is also the very ample case of Theorem~\ref{thm:B}.
\section{Hypersurfaces}\label{sec:hyp}

In this section we deduce the categorical Torelli theorems for smooth Fano hypersurfaces from Theorems~\ref{thm:A} and~\ref{thm:B} and from Theorem~\ref{thm:general}. All of them are known; here they appear as special cases of the general theorems.

Let $X\subset\PP^{n+1}$ be a smooth hypersurface of degree $d$ with $3\le d\le n+1$, put $H=\cO_X(1)$ and $r=n+2-d$, so that $\omega_X^{-1}\simeq H^{\otimes r}$, and write $\Ku(X)=\Ku(X,\cO_X(1))$. For $n\ge3$, $\cO_X(1)$ generates $\Pic(X)$ by the Lefschetz theorem, so this is the canonical Kuznetsov component. The line bundle $\cO_X(1)$ is very ample and the ambient linear forms give the complete linear system $|\cO_X(1)|$, so an isomorphism $(X,\cO_X(1))\simeq(Y,\cO_Y(1))$ is a projective equivalence. We also use the periodicity
\begin{equation}\label{eq:periodic}
\Rot_X^{\,d}\simeq[2],
\end{equation}
which follows from the identification of $\Ku(X)$ with the category of graded matrix factorizations of the equation of $X$, under which the rotation functor corresponds to the grading shift \cite{Orl09}, \cite[Theorem~2.2]{LZ25}; see also \cite[Theorem~4.6]{Pir24}.

\begin{corollary}[{\cite{HR19,Pir24,BMMS12}}]\label{cor:hypmarked}
Let $n\ge3$ and $3\le d\le n+1$, and let $X,Y\subset\PP^{n+1}$ be smooth hypersurfaces of degree $d$. If $\Ku(X)\simeq\Ku(Y)$ by an equivalence commuting with the rotation functors, then $X$ and $Y$ are projectively equivalent.
\end{corollary}

\begin{proof}
For $d\ge4$ we have $r=n+2-d\le n-2$, and since $\cO_X(1)$ and $\cO_Y(1)$ are very ample, Theorem~\ref{thm:A} gives $(X,\cO_X(1))\simeq(Y,\cO_Y(1))$. For $d=3$ and $n\ge4$ we have $n-r=1$, and the geometric map $\theta_1$ is an isomorphism by the computation of natural transformations of graded matrix factorizations \cite[Lemma~4.4]{Pir24}, see also \cite[Theorem~5.39]{BFK14} and \cite[Theorem~3.5]{LRZ26}; Theorem~\ref{thm:general} applies in the case $n-r=1$. For cubic threefolds, Theorem~\ref{thm:dP3} gives the stronger statement that $\Ku(X)$ alone determines $X$. In each case the polarized isomorphism is a projective equivalence.
\end{proof}

\begin{corollary}[{\cite{LZ25}}]\label{cor:hypcoprime}
Let $n\ge3$ and $3\le d\le n+1$ with $\gcd(n+2,d)=1$, and let $X,Y\subset\PP^{n+1}$ be smooth hypersurfaces of degree $d$. If $\Ku(X)\simeq\Ku(Y)$, then $X$ and $Y$ are projectively equivalent.
\end{corollary}

\begin{proof}
Since $\gcd(r,d)=\gcd(n+2,d)=1$, choose $a,b$ with $ar+bd=1$. By Proposition~\ref{prop:serre} and \eqref{eq:periodic},
\[
\Rot_X=\Rot_X^{\,ar+bd}\simeq\mathsf S_{\Ku(X)}^{-a}[an+2b],
\]
and similarly for $Y$, so every equivalence $\Ku(X)\simeq\Ku(Y)$ commutes with the rotation functors. Corollary~\ref{cor:hypmarked} applies.
\end{proof}

\begin{corollary}[{\cite{LZ25}}]\label{cor:hypindexone}
Let $n\ge3$ and let $X\subset\PP^{n+1}$ be a smooth hypersurface of degree $n+1$. If $Y$ is a smooth Fano variety of dimension at least three with globally generated anticanonical bundle and $\cO_X^{\perp}\simeq\cO_Y^{\perp}$, then $Y\simeq X$, and the anticanonical morphism of $Y$ is an embedding as a hypersurface of degree $n+1$ in $\PP^{n+1}$ projectively equivalent to $X$.
\end{corollary}

\begin{proof}
Here $-K_X=\cO_X(1)$ is very ample, and the anticanonical embedding of $X$ is the given one. By Lemma~\ref{lem:vadual} its dual is a hypersurface, and Theorem~\ref{thm:B} gives an isomorphism $X\simeq Y$ compatible with the anticanonical morphisms; hence that of $Y$ is an embedding with the same image.
\end{proof}

\begin{remark}\label{rem:hyp}
\leavevmode
\begin{enumerate}[label=(\alph*)]
\item \emph{Cubic hypersurfaces.} Corollary~\ref{cor:hypmarked} contains the theorem of Bernardara, Macr\`\i, Mehrotra and Stellari for cubic threefolds \cite{BMMS12} and the theorems with rotation functors of Huybrechts--Rennemo for cubic fourfolds \cite{HR19} and of Pirozhkov in dimension at least four \cite{Pir24}. For cubic fourfolds the rotation functor cannot be removed in general \cite{Per21}. The theorem of Lin--Zhang for arbitrary equivalences when $\gcd(n+2,3)=1$ \cite{LZ25}, which includes cubic surfaces, is Corollary~\ref{cor:hypcoprime} for $n\ge3$.
\item \emph{Outside the coprime range.} The theorem for arbitrary equivalences when $d\nmid n+2$ and $(d,n+2)\ne(4,4k+2)$ is due to Lin, Rennemo and Zhang \cite[Theorem~3.12 and Corollary~3.13]{LRZ26}, who reconstruct $X$ from the intrinsic power $\Rot_X^{\,g}$, $g=\gcd(n+2,d)$.
\item \emph{Weighted hypersurfaces.} The theorems with rotation functors, and for arbitrary equivalences in the coprime case, for quasismooth Fano weighted hypersurfaces are due to \cite{Lin26,LRZ26,LZZ26}. Such a hypersurface is in general a stack, and when it is a smooth variety its weight-one polarization may have projective space as image, with empty dual. Proposition~\ref{prop:veronese} treats the anticanonical residual categories of some smooth cyclic covers of projective spaces.
\end{enumerate}
\end{remark}

\section{Prime Fano threefolds of index one}\label{sec:prime}

A smooth Fano threefold $X$ is \emph{prime of index one} if $\Pic(X)=\ZZ\cdot K_X$. Its genus $g$ is defined by $(-K_X)^3=2g-2$; then $h^0(X,-K_X)=g+2$, and $2\le g\le12$, $g\ne11$. We use the following facts \cite[Chapter~4]{IP99}. The anticanonical bundle is very ample unless $X$ is \emph{hyperelliptic}; the hyperelliptic ones are the sextic double solids, of genus two, and the double covers of a smooth quadric threefold branched along its intersection with a quartic, of genus three. Since $X$ has index one, its canonical Kuznetsov component is $\Ku(X)=\cO_X^{\perp}$. For $g\in\{6,7,8,9,10,12\}$ there is moreover a semiorthogonal decomposition
\begin{equation}\label{eq:Mukai}
\Db(X)=\langle\cK_X,\cE_X,\cO_X\rangle,
\end{equation}
where $\cE_X$ is an exceptional vector bundle, the Mukai bundle; the smaller category $\cK_X$ is the one usually called the Kuznetsov component of $X$ in the literature \cite{Kuz14,JLZ26}, and $\cO_X^{\perp}=\langle\cK_X,\cE_X\rangle$. Both results of this section are consequences of Theorem~\ref{thm:B} and Corollary~\ref{cor:vaPerry}.

\begin{corollary}[Perry's conjecture for prime Fano threefolds]\label{cor:primePerry}
Let $X$ be a smooth prime Fano threefold of index one and genus $g\ge3$, and let $Y$ be a smooth Fano variety of dimension at least three with globally generated anticanonical bundle. If $\cO_X^{\perp}\simeq\cO_Y^{\perp}$, then $X\simeq Y$. This applies to every family of genus at least four, to the quartic threefolds, and to the hyperelliptic threefolds of genus three.
\end{corollary}

\begin{proof}
If $X$ is not hyperelliptic, $-K_X$ is very ample, and Lemma~\ref{lem:vadual} and Theorem~\ref{thm:B} apply. The remaining case, the double covers of a quadric threefold of genus three, is Corollary~\ref{cor:hyperelliptic}.
\end{proof}

\begin{corollary}[{\cite{JLZ26,JLLZ24}}]\label{thm:refined}
Let $X$ and $Y$ be smooth prime Fano threefolds of index one and genus $g\in\{6,7,8,9,10,12\}$, with the decompositions \eqref{eq:Mukai}. An equivalence $\Phi\colon\cK_X\simeq\cK_Y$ with $\Phi(i^!\cE_X)\simeq i^!\cE_Y$, up to an autoequivalence of $\cK_Y$ and a shift, induces an isomorphism $X\simeq Y$.
\end{corollary}

\begin{proof}
The category $\cO_X^{\perp}=\langle\cK_X,\cE_X\rangle$ has gluing object $i^!\cE_X$, since the right adjoint of $\cK_X\subset\Db(X)$ agrees on $\cO_X^{\perp}$ with the right adjoint of $\cK_X\subset\cO_X^{\perp}$. Corollary~\ref{cor:gluingobject} extends $\Phi$ to an equivalence $\cO_X^{\perp}\simeq\cO_Y^{\perp}$. In these genera the anticanonical bundles are very ample, and Corollary~\ref{cor:vaPerry} gives $X\simeq Y$.
\end{proof}

\section{Further examples}\label{sec:examples}

In Sections~\ref{ss:CI}--\ref{ss:dP3} the relevant polarizations are very ample; Sections~\ref{ss:cyclic} and~\ref{ss:veronese} treat finite polarization maps of higher degree.

\subsection{Complete intersections}\label{ss:CI}

\begin{theorem}\label{thm:CI}
Fix $n\ge3$ and a multidegree $(d_1,\dots,d_c)$ with all $d_j\ge2$, and put $s=\sum_j(d_j-1)\le n$. Let $X,Y\subset\PP^{n+c}$ be smooth complete intersections of this type.
\begin{enumerate}[label=\textup{(\arabic*)}]
\item If $s\ge3$, an equivalence $\Ku(X)\simeq\Ku(Y)$ commuting with the rotation functors implies that $X$ and $Y$ are projectively equivalent.
\item If $s=n$, so that $X$ has index one, an equivalence $\cO_X^{\perp}\simeq\cO_Y^{\perp}$ implies that $X$ and $Y$ are projectively equivalent.
\end{enumerate}
\end{theorem}

\begin{proof}
By the Grothendieck--Lefschetz theorem \cite[Section~3.1]{Laz04}, $\Pic(X)=\ZZ\cdot\cO_X(1)$, and by adjunction $-K_X=\cO_X(n+1-s)$; so $\cO_X(1)$ is the fundamental divisor, $i_X=n+1-s$, and it is very ample. Since all $d_j\ge2$, the Koszul resolution shows that $X$ is linearly normal and nondegenerate, so $H^0(X,\cO_X(1))=H^0(\PP^{n+c},\cO(1))$. The condition $i_X\le n-2$ is $s\ge3$, and Theorem~\ref{thm:A} gives (1); Corollary~\ref{cor:vaPerry} gives (2). The isomorphisms identify the complete embeddings, hence are projective equivalences.
\end{proof}

\subsection{Gushel--Mukai varieties}

\begin{proposition}\label{prop:GM}
Let $X$ and $Y$ be smooth ordinary Gushel--Mukai varieties of dimension $n\in\{3,4,5\}$, that is, smooth intersections $\Gr(2,5)\cap\PP^{n+4}\cap Q$ with a quadric $Q$. Then an equivalence $\Ku(X)\simeq\Ku(Y)$ of canonical Kuznetsov components commuting with the rotation functors implies $X\simeq Y$; for $n=3$ an equivalence $\cO_X^{\perp}\simeq\cO_Y^{\perp}$ suffices.
\end{proposition}

\begin{proof}
The Pl\"ucker class $H$ is very ample, and $-K_X=(n-2)H$ \cite{DK18}. Since $\Pic(X)$ is torsion free and $H^n=10$ is not divisible by $k^n$ for any $k\ge2$, the class $H$ is primitive, so it is the fundamental divisor and $i_X=n-2$. Theorem~\ref{thm:A} applies, and for $n=3$ so does Corollary~\ref{cor:vaPerry}.
\end{proof}

\subsection{Hypersurfaces in products of projective spaces}

\begin{proposition}\label{prop:products}
Let $X,Y\subset\PP^a\times\PP^b$ be smooth hypersurfaces of bidegree $(d_1,d_2)$ with $1\le d_1\le a$, $1\le d_2\le b$ and $a+b\ge4$. Then $\cO_X^{\perp}\simeq\cO_Y^{\perp}$ implies $X\simeq Y$. Moreover, if $\iota=\gcd(a+1-d_1,b+1-d_2)\le a+b-3$, an equivalence of canonical Kuznetsov components commuting with the rotation functors implies $X\simeq Y$.
\end{proposition}

\begin{proof}
By adjunction $-K_X=\cO_X(a+1-d_1,b+1-d_2)$ with both entries positive, which is very ample, and $\dim X=a+b-1\ge3$, so Corollary~\ref{cor:vaPerry} applies. By the Lefschetz theorem $\Pic(X)=\Pic(\PP^a\times\PP^b)$, so the fundamental divisor is $\cO_X((a+1-d_1)/\iota,(b+1-d_2)/\iota)$, which is very ample, and $i_X=\iota$; Theorem~\ref{thm:A} applies when $\iota\le\dim X-2$.
\end{proof}

\subsection{Del Pezzo threefolds}\label{ss:dP3}
For del Pezzo threefolds of degree one the categorical Torelli theorem is due to Lin, Rennemo and Zhang \cite{LRZ26}, via Hochschild--Serre algebras; see Section~\ref{ss:comparison}. For a smooth del Pezzo threefold $Y_d$ of Picard rank one and degree $d\in\{2,3,4\}$ one has $-K_{Y_d}=2H$ and $\Db(Y_d)=\langle\Ku(Y_d),\cO,\cO(H)\rangle$. Mutating $\cO(H)$ gives
\begin{equation}\label{eq:dP3}
\Db(Y_d)=\langle\Ku(Y_d),\mathcal Q_{Y_d},\cO_{Y_d}\rangle,\qquad\mathcal Q_{Y_d}=L_{\cO}\cO(H)[-1],
\end{equation}
where $\mathcal Q_{Y_d}$ is the kernel of the evaluation map of $H$. By \cite[Theorem~7.1]{FLZ23}, the object $i^!\mathcal Q_{Y_d}$ is preserved by every equivalence of Kuznetsov components, up to autoequivalences of the target and shifts.

\begin{theorem}\label{thm:dP3}
Let $Y_d$ and $Y'_d$ be smooth del Pezzo threefolds of Picard rank one and degree $d\in\{2,3,4\}$. Then $\Ku(Y_d)\simeq\Ku(Y'_d)$ implies $Y_d\simeq Y'_d$.
\end{theorem}

\begin{proof}
By \eqref{eq:dP3}, $\cO_{Y_d}^{\perp}=\langle\Ku(Y_d),\mathcal Q_{Y_d}\rangle$, and by \cite[Theorem~7.1]{FLZ23} and Corollary~\ref{cor:gluingobject}, $\cO_{Y_d}^{\perp}\simeq\cO_{Y'_d}^{\perp}$. For $d=3,4$, $H$ and hence $-K=2H$ is very ample. For $d=2$, $Y_2$ is $y^2=F_4(x_0,\dots,x_3)$ in $\PP(1,1,1,1,2)$; it avoids the singular point of the weighted projective space, and the second Veronese subring of its coordinate ring is generated in degree one by the quadratic monomials in the $x_i$ and by $y$, so $2H$ is very ample. Corollary~\ref{cor:vaPerry} applies.
\end{proof}

\subsection{Cyclic covers of quadrics}\label{ss:cyclic}
Theorems~\ref{thm:A} and~\ref{thm:B} apply to globally generated polarizations which are not very ample, provided the image of the polarization map has a hypersurface as dual. We first record when a Veronese re-embedding has hypersurface dual.

\begin{lemma}\label{lem:veronesedual}
Let $M\subset\PP^N$ be a smooth projective variety embedded by a very ample line bundle $L$. For every $a\ge2$, the projective dual of the embedding of $M$ by the complete linear system $|L^{\otimes a}|$ is a hypersurface.
\end{lemma}

\begin{proof}
Let $p\in M$, choose $s_0\in H^0(M,L)$ with $s_0(p)\ne0$, and sections $s_1,\dots,s_m$ such that the functions $s_i/s_0$ vanish at $p$ and span the cotangent space there. The products $s_0^a$, $s_is_0^{a-1}$ and $s_is_js_0^{a-2}$ span all second jets at $p$, so among the sections of $L^{\otimes a}$ with vanishing first jet at $p$ there is one with nondegenerate quadratic term. At the corresponding conormal pair the differential of the conormal projection to the dual projective space is injective, by the Hessian computation in the proof of Lemma~\ref{lem:tangent}, so the dual has dimension one less than the dual projective space.
\end{proof}

\begin{corollary}\label{cor:cyclicquadric}
Let $Q\subset\PP^{n+1}$ be a smooth quadric of dimension $n\ge3$, and let $\pi\colon X\to Q$ be the cyclic cover of degree $d\ge2$ associated with $\cO_Q(b)$ and a smooth branch divisor in $|\cO_Q(db)|$. Assume that $b\ge2$ and $r=n-(d-1)b\ge1$, and put $H=\pi^*\cO_Q(1)$.
\begin{enumerate}[label=\textup{(\arabic*)}]
\item The pair $(\Ku(X,H),\Rot_{X,H})$ determines $(X,H)$ among smooth Fano $n$-folds with a globally generated Fano polarization of index $r$.
\item The category $\cO_X^{\perp}$ determines $X$ among all smooth Fano varieties of dimension at least three with globally generated anticanonical bundle.
\end{enumerate}
\end{corollary}

\begin{proof}
The cyclic algebra structure and the ramification formula give
\[
\pi_*\cO_X=\bigoplus_{j=0}^{d-1}\cO_Q(-jb),\qquad-K_X=\pi^*\cO_Q\bigl(n-(d-1)b\bigr)=rH;
\]
equivalently, $X$ is the hypersurface $z^d=s$ in the total space of $\cO_Q(b)$. A smooth branch divisor makes $X$ smooth, and $r>0$ makes it Fano. Since $b\ge2$, $H^0(X,H)=H^0(Q,\cO_Q(1))$, so the morphism defined by $|H|$ is $\pi$ followed by the embedding of $Q$; its image is $Q$, whose dual is a quadric hypersurface. Moreover $1\le r\le n-2$, so Theorem~\ref{thm:A} proves (1).

The anticanonical bundle $rH$ is globally generated. If $r<b$, then $H^0(X,rH)=H^0(Q,\cO_Q(r))$, so the anticanonical image is $Q$ in its complete embedding by $\cO_Q(r)$; its dual is a hypersurface, for $r=1$ because $Q$ is a smooth quadric and for $r\ge2$ by Lemma~\ref{lem:veronesedual}. If $r\ge b$, then $rH$ is very ample: on the chart where a linear coordinate $x_0$ of $\PP^{n+1}$ does not vanish, the sections of $rH$ include all monomials of degree $r$ in the coordinates and $zx_0^{r-b}$, and their ratios to $x_0^r$ give the affine coordinates $x_i/x_0$ and the fibre coordinate $z/x_0^b$, so the complete anticanonical map is a closed immersion; Lemma~\ref{lem:vadual} applies. In all cases Theorem~\ref{thm:B} proves (2).
\end{proof}

\begin{corollary}\label{cor:hyperelliptic}
Let $X$ be a smooth double cover of a smooth quadric threefold branched along a smooth quartic section. Then $\cO_X^{\perp}$ determines $X$ among all smooth Fano varieties of dimension at least three with globally generated anticanonical bundle. In particular the hyperelliptic prime Fano threefolds of genus three satisfy Perry's conjecture.
\end{corollary}

\begin{proof}
Take $n=3$, $d=2$ and $b=2$ in Corollary~\ref{cor:cyclicquadric}. Then $-K_X=H$, $H^3=4$ and $h^0(X,H)=5$, and the anticanonical morphism is two-to-one onto the quadric. These are the hyperelliptic Fano threefolds of genus three in the classification \cite{IP99}, with weighted complete-intersection equations $q_2(x)=0$, $z^2=f_4(x)$ in $\PP(1^5,2)$.
\end{proof}

\begin{example}\label{ex:triplecover}
Let $X$ be a cyclic triple cover of a smooth quadric fivefold branched along a smooth sextic section. Here $n=5$, $d=3$, $b=2$ and $r=1$, so $-K_X=\pi^*\cO_Q(1)$ is globally generated, the anticanonical morphism has degree three onto $Q$, and Corollary~\ref{cor:cyclicquadric} shows that $X$ is determined by $\cO_X^{\perp}$.
\end{example}

\subsection{Veronese anticanonical images}\label{ss:veronese}

\begin{proposition}\label{prop:veronese}
Let $X$ be the smooth cyclic cover of $\PP^n$, $n\ge3$, of degree $d\ge2$ associated with $\cO_{\PP^n}(b)$ and a smooth branch divisor of degree $db$, and put $a=n+1-(d-1)b$. If $a\ge2$, then $\cO_X^{\perp}$ determines $X$ among all smooth Fano varieties of dimension at least three with globally generated anticanonical bundle.
\end{proposition}

\begin{proof}
Write $H=\pi^*\cO_{\PP^n}(1)$. As in the proof of Corollary~\ref{cor:cyclicquadric}, $-K_X=aH$ and $\pi_*\cO_X=\bigoplus_{j=0}^{d-1}\cO(-jb)$. If $2\le a<b$, the anticanonical morphism is the cover followed by the $a$-th Veronese embedding of $\PP^n$, whose dual is a hypersurface by Lemma~\ref{lem:veronesedual}. If $a\ge b$, the monomials of degree $a$ and the fibre coordinate show as before that $aH$ is very ample, and Lemma~\ref{lem:vadual} applies. In both cases Theorem~\ref{thm:B} proves the statement.
\end{proof}

\begin{example}\label{ex:decic}
For the smooth double cover $X$ of $\PP^6$ branched along a decic we have $d=2$, $b=5$ and $a=2$, so $-K_X=2H$ and the anticanonical morphism is two-to-one onto the Veronese variety $\nu_2(\PP^6)$. Proposition~\ref{prop:veronese} shows that $X$ is determined by $\cO_X^{\perp}$. Theorem~\ref{thm:C} does not apply here: $-qK_X=2qH$ is very ample only for $2q\ge5$, whereas $q\ell\le n-1=5$ with $\ell\ge2$ forces $q\le2$.
\end{example}

\subsection{Comparison with known categorical Torelli theorems}\label{ss:comparison}
Table~\ref{tab:comparison} compares the main applications with the literature.

\begin{table}[ht]
\centering
\small
\renewcommand{\arraystretch}{1.2}
\begin{tabular}{>{\raggedright\arraybackslash}p{3.3cm}>{\raggedright\arraybackslash}p{3.3cm}>{\raggedright\arraybackslash}p{2.5cm}>{\raggedright\arraybackslash}p{3.7cm}}
\toprule
Family & Categorical input & Earlier result & Deduction here\\
\midrule
Fano hypersurfaces, $d\ge4$ & $\Ku(X)$ with rotation; the category alone in the coprime cases & \cite{HR19,Pir24,LZ25} & Corollaries~\ref{cor:hypmarked} and~\ref{cor:hypcoprime}\\
Cubic hypersurfaces, $n\ge4$ & $\Ku(X)$ with rotation & \cite{HR19,Pir24} & Theorem~\ref{thm:general} with $n-r=1$\\
Cubic threefolds, quartic double solids, intersections of two quadrics in $\PP^5$ & The usual Kuznetsov component & \cite{BMMS12,APR22,FLZ23} & Gluing object of \cite{FLZ23}, then Corollary~\ref{cor:vaPerry}\\
Prime Fano threefolds, $g\ge6$ & Smaller component with its gluing object & \cite{JLLZ24,JLZ26} & Corollaries~\ref{thm:refined} and~\ref{cor:vaPerry}\\
Prime Fano threefolds, $g\ge3$ & $\cO_X^{\perp}$ & & Corollary~\ref{cor:primePerry}\\
Cyclic covers of quadrics and of $\PP^n$ & $\cO_X^{\perp}$; $\Ku(X,H)$ with rotation & & Corollary~\ref{cor:cyclicquadric}, Proposition~\ref{prop:veronese}\\
\bottomrule
\end{tabular}
\caption{Known categorical Torelli theorems and the finite-map applications.}
\label{tab:comparison}
\end{table}

For sextic double solids and for Veronese double cones, the del Pezzo threefolds of degree one, categorical Torelli theorems for very general members were proved by Dell, Jacovskis and Rota \cite[Theorems~6.4 and~6.5]{DJR23} using Hodge theory and the covering involution. For del Pezzo threefolds of degree one, which are weighted hypersurfaces of degree six in $\PP(1,1,1,1,2,3)$, the categorical Torelli theorem was proved by Lin, Rennemo and Zhang \cite{LRZ26} using Hochschild--Serre algebras. Their anticanonical images are $\PP^3$, with empty dual, and the cone over $\nu_2(\PP^2)$, whose dual has codimension two, since a tangent hyperplane must contain the vertex and cut out a singular conic on $\PP^2$.

\appendix

\section{Reconstruction from the equations of pluri-anticanonical embeddings}\label{app:equations}

In this appendix we prove Theorem~\ref{thm:C}. The argument only uses the Hochschild--Serre algebra of $\cO_X^{\perp}$, the kernel estimate of Proposition~\ref{prop:kernel}, and the multiplicativity of the comparison maps; it is independent of the noncommutative divisors of Sections~\ref{sec:ncdiv}--\ref{sec:main}. We write $X'$ for the second variety, called $Y$ in the statement of Theorem~\ref{thm:C}. The proof is a chain of identifications, each recorded by a commutative diagram:
\begin{enumerate}[label=(\arabic*)]
\item an equivalence $\cO_X^{\perp}\simeq\cO_{X'}^{\perp}$ identifies the Hochschild--Serre algebras, as bigraded algebras (Section~\ref{ss:A-HS});
\item for each variety, multiplication by pluri-anticanonical sections maps its section ring to its Hochschild--Serre algebra by a homomorphism of graded rings, which is injective in low degrees (Sections~\ref{ss:A-setting}--\ref{ss:A-mult});
\item hence the homogeneous coordinate rings of the embeddings by $|-qK_X|$ and $|-qK_{X'}|$ are identified in low degrees, together with the kernels of the multiplication maps, that is, the low-degree equations (Sections~\ref{ss:A-rings} and~\ref{ss:A-proofC});
\item when these equations generate the ideals, the identification passes to $\Proj$ of the coordinate rings, that is, to the varieties (Section~\ref{ss:A-proofC}).
\end{enumerate}

\subsection{The Hochschild--Serre algebra}\label{ss:A-HS}
Let $\cA$ be a smooth proper enhanced category with Serre functor $\mathsf S=\mathsf S_\cA$. We work in the homotopy categories of enhanced functors, in which composition is a bifunctor; thus vertical composition $\circ$ and horizontal composition $*$ of natural transformations satisfy the interchange law
\begin{equation}\label{eq:interchange}
(\beta'\circ\beta)*(\alpha'\circ\alpha)=(\beta'*\alpha')\circ(\beta*\alpha),
\end{equation}
where, for $\beta\colon F\to F'$ and $\alpha\colon G\to G'$, we have $\beta*\alpha=F'(\alpha)\circ\beta_G=\beta_{G'}\circ F(\alpha)\colon F\circ G\to F'\circ G'$.

The Hochschild--Serre algebra of $\cA$ \cite{BFK23,LZ25} is
\[
\HS^{\bullet,\bullet}(\cA)=\bigoplus_{p,q\in\ZZ}\HS^{p,q}(\cA),\qquad\HS^{p,q}(\cA)=\Hom\bigl(\id_\cA,\mathsf S^q[p]\bigr),
\]
with the product
\begin{equation}\label{eq:HSproduct}
u\cdot v=\mathsf S^q[p](v)\circ u=u*v\in\HS^{p+p',q+q'}(\cA)\qquad\bigl(u\in\HS^{p,q}(\cA),\ v\in\HS^{p',q'}(\cA)\bigr),
\end{equation}
where $\mathsf S^q[p]\circ\mathsf S^{q'}[p']=\mathsf S^{q+q'}[p+p']$. It is an associative bigraded algebra with unit $\id\in\HS^{0,0}(\cA)$.

Let $\Phi\colon\cA\xrightarrow{\sim}\cB$ be an equivalence, and fix an isomorphism $\Phi\circ\mathsf S_\cA\simeq\mathsf S_\cB\circ\Phi$, which exists since Serre functors are intrinsic. Conjugation by $\Phi$, followed by this isomorphism, defines linear isomorphisms $\Phi_*\colon\HS^{p,q}(\cA)\to\HS^{p,q}(\cB)$, and these are compatible with products \cite[Theorem~3.4]{LZ25}: for all bidegrees, the square
\begin{equation}\label{eq:diagHSinv}
\begin{tikzcd}
\HS^{p,q}(\cA)\otimes\HS^{p',q'}(\cA)\arrow[r,"\cdot"]\arrow[d,"\Phi_*\otimes\Phi_*"',"\wr"] & \HS^{p+p',q+q'}(\cA)\arrow[d,"\Phi_*","\wr"']\\
\HS^{p,q}(\cB)\otimes\HS^{p',q'}(\cB)\arrow[r,"\cdot"] & \HS^{p+p',q+q'}(\cB)
\end{tikzcd}
\end{equation}
commutes. Changing the isomorphism $\Phi\mathsf S_\cA\simeq\mathsf S_\cB\Phi$ by a scalar $\lambda\in\CC^*$ multiplies $\Phi_*$ by $\lambda^q$ on $\HS^{p,q}$; this does not affect any of the subspaces considered below.

\subsection{Setting and the comparison maps}\label{ss:A-setting}
From now on, $X$ is a smooth Fano variety of dimension $n$, $L=\omega_X^{-1}$, and $\cC_X=\cO_X^{\perp}$, with inclusion $i$ and left adjoint $\pi$. We use the notation of Sections~\ref{ss:Ku} and~\ref{ss:kernels} for the anticanonical polarization $(L,1)$ and $\cB_X=\langle\cO_X\rangle$: $T=(-)\otimes L$, $\Rot=\pi Ti$, $\cH_m=\RHom(\id_{\cC_X},\Rot^m)$, and $\HS^{p,q}_X=\HS^{p,q}(\cC_X)$. By Proposition~\ref{prop:serre} with $r=1$ there is an isomorphism $\Rot\simeq\mathsf S^{-1}[n]$, which we fix; its powers give isomorphisms $\Rot^m\simeq\mathsf S^{-m}[mn]$ compatible with composition. Hence, for all $p\in\ZZ$ and $m\ge0$,
\begin{equation}\label{eq:HSdegrees}
\HS^{p,-m}_X=\Hom\bigl(\id,\Rot^m[p-mn]\bigr)=H^{p-mn}(\cH_m),\qquad\HS^{mn,-m}_X=\Hom(\id,\Rot^m)=H^0(\cH_m).
\end{equation}
Thus sections of $L^m=\omega_X^{-m}$ give classes in bidegree $(mn,-m)$: the comparison map of Section~\ref{ss:kernels} is
\[
\theta_m\colon H^0(X,-mK_X)\longrightarrow\HS^{mn,-m}_X,
\]
and Proposition~\ref{prop:kernel} with $r=1$ gives
\begin{equation}\label{eq:HSbound}
\cH_m\in D^{\ge\min\{0,n-1-m\}}(\CC),\qquad\theta_m\ \text{is bijective for }m\le n-2\text{ and injective for }m\le n-1.
\end{equation}
On the subalgebra $\bigoplus_{m\ge0}\HS^{mn,-m}_X$, the product \eqref{eq:HSproduct} is $a\cdot b=\Rot^e(b)\circ a$ for $a\in\HS^{en,-e}_X$.

We write $\mathfrak F_X$ for the homotopy category of enhanced endofunctors of $\Perf(X)$, equivalently of perfect kernels on $X\times X$ under convolution \cite[Theorem~8.15]{Toe07}, \cite[Theorem~6.4]{LS16}. Let $P=i\pi$, with unit $\eta\colon\id\to P$ and multiplication $\mu=i\varepsilon\pi\colon P\circ P\to P$, where $\varepsilon\colon\pi i\to\id$ is the counit. Since $i$ is fully faithful, $\varepsilon$ and $\mu$ are isomorphisms, and
\begin{equation}\label{eq:monad}
\mu\circ(\eta*\id_P)=\id_P=\mu\circ(\id_P*\eta),\qquad\mu\circ(\eta*\eta)=\eta;
\end{equation}
the last identity follows from the first, since $\eta*\eta=(\eta*\id_P)\circ\eta$ by \eqref{eq:interchange}. Recall from \eqref{eq:Qm} that $i\Rot^m\pi=PTPT\cdots TP$ with $m$ factors $T$, that $u_m\colon T^m\to i\Rot^m\pi$ inserts $\eta$ at each of the $m+1$ factors $P$, and that $\theta_m(s)$ is defined by $u_m\circ s$, where $s\colon\id\to T^m$ is multiplication by $s$.

\subsection{Multiplicativity of the comparison maps}\label{ss:A-mult}
For a transformation $a$ between endofunctors of $\cC_X$ write $E(a)=i\,a\,\pi$. Since $\pi i=\id$ through $\varepsilon$, the identity $\Rot^e\Rot^f=\Rot^{e+f}$ gives the isomorphism
\begin{equation}\label{eq:junction}
\iota_{e,f}\colon i\Rot^e\pi\circ i\Rot^f\pi=(PT\cdots TP)\circ(PT\cdots TP)\xrightarrow{\ \sim\ }PT\cdots TP=i\Rot^{e+f}\pi,
\end{equation}
which is $\mu$ applied to the two adjacent factors $P$ at the junction.

\begin{lemma}\label{lem:kerneldesc}
For $m\ge0$ the map
\[
\rho_m\colon\HS^{mn,-m}_X\longrightarrow\Hom_{\mathfrak F_X}(\id,i\Rot^m\pi),\qquad a\longmapsto E(a)\circ\eta,
\]
is bijective, and $\rho_m(\theta_m(s))=u_m\circ s$ for $s\in H^0(X,-mK_X)$. Moreover, for $a\in\HS^{en,-e}_X$ and $b\in\HS^{fn,-f}_X$,
\begin{equation}\label{eq:productkernel}
\rho_{e+f}(a\cdot b)=\iota_{e,f}\circ\bigl(\rho_e(a)*\rho_f(b)\bigr).
\end{equation}
\end{lemma}

\begin{proof}
By \eqref{eq:HSdegrees}, $\HS^{mn,-m}_X=\Hom(\id,\Rot^m)$. The map $a\mapsto E(a)$ is the bijection $\Hom(\id,\Rot^m)\simeq\Hom_{\mathfrak F_X}(P,i\Rot^m\pi)$ of \eqref{eq:Hm}, and restriction along $\eta$ is the bijection of Lemma~\ref{lem:diagonal}; this proves the first assertion, and the second is the definition of $\theta_m$. By \eqref{eq:HSproduct}, $a\cdot b=a*b$, and $E(a*b)\circ\mu=\iota_{e,f}\circ(E(a)*E(b))$ as transformations $P\circ P\to i\Rot^{e+f}\pi$, by naturality of $\varepsilon$. Composing with $\eta*\eta$ and using $\mu\circ(\eta*\eta)=\eta$ from \eqref{eq:monad} and \eqref{eq:interchange}, we obtain
\[
\rho_{e+f}(a\cdot b)=E(a*b)\circ\eta=\iota_{e,f}\circ(E(a)*E(b))\circ(\eta*\eta)=\iota_{e,f}\circ\bigl((E(a)\circ\eta)*(E(b)\circ\eta)\bigr).\qedhere
\]
\end{proof}

\begin{lemma}\label{lem:mult}
For all $e,f\ge0$ the square
\begin{equation}\label{eq:diagmult}
\begin{tikzcd}
H^0(X,-eK_X)\otimes H^0(X,-fK_X)\arrow[r,"\text{mult}"]\arrow[d,"\theta_e\otimes\theta_f"'] & H^0(X,-(e+f)K_X)\arrow[d,"\theta_{e+f}"]\\
\HS^{en,-e}_X\otimes\HS^{fn,-f}_X\arrow[r,"\cdot"] & \HS^{(e+f)n,-(e+f)}_X
\end{tikzcd}
\end{equation}
commutes, that is, $\theta_e(s)\cdot\theta_f(t)=\theta_{e+f}(st)$. Moreover $\theta_0(1)=\id$.
\end{lemma}

\begin{proof}
By Lemma~\ref{lem:kerneldesc} and \eqref{eq:interchange},
\[
\rho_{e+f}\bigl(\theta_e(s)\cdot\theta_f(t)\bigr)=\iota_{e,f}\circ\bigl((u_e\circ s)*(u_f\circ t)\bigr)=\iota_{e,f}\circ(u_e*u_f)\circ(s*t).
\]
The horizontal composite $s*t\colon\id\to T^e\circ T^f=T^{e+f}$ is $T^e(t)\circ s$, that is, multiplication by $st$. The transformation $u_e*u_f$ inserts $\eta$ at each of the $e+1$ factors $P$ of $i\Rot^e\pi$ and at each of the $f+1$ factors $P$ of $i\Rot^f\pi$; the map $\iota_{e,f}$ replaces the two adjacent factors $P\circ P$ at the junction by $P$ through $\mu$, and $\mu\circ(\eta*\eta)=\eta$ by \eqref{eq:monad}. Hence $\iota_{e,f}\circ(u_e*u_f)=u_{e+f}$, and
\[
\rho_{e+f}\bigl(\theta_e(s)\cdot\theta_f(t)\bigr)=u_{e+f}\circ(st)=\rho_{e+f}\bigl(\theta_{e+f}(st)\bigr).
\]
Since $\rho_{e+f}$ is injective, the square commutes. Finally $\rho_0(\theta_0(1))=u_0=\eta=\rho_0(\id)$.
\end{proof}

\subsection{The dimension}\label{ss:A-dim}

\begin{lemma}\label{lem:dimq}
Let $q\ge1$ with $q\le n-1$ and $H^0(X,-qK_X)\ne0$. Then
\[
qn=\min\bigl\{p\in\ZZ:\HS^{p,-q}_X\ne0\bigr\}.
\]
In particular, if $X'$ is a smooth Fano variety satisfying the same hypotheses for the same $q$, and $\cO_X^{\perp}\simeq\cO_{X'}^{\perp}$, then $\dim X=\dim X'$.
\end{lemma}

\begin{proof}
By \eqref{eq:HSdegrees}, $\HS^{p,-q}_X=H^{p-qn}(\cH_q)$. Since $q\le n-1$, \eqref{eq:HSbound} gives $\cH_q\in D^{\ge0}(\CC)$, so $\HS^{p,-q}_X=0$ for $p<qn$; and $\HS^{qn,-q}_X$ contains $\theta_q(H^0(X,-qK_X))\ne0$, since $\theta_q$ is injective. For the second assertion, an equivalence identifies the bigraded vector spaces $\HS^{\bullet,\bullet}$ (Section~\ref{ss:A-HS}), so $q\dim X=q\dim X'$.
\end{proof}

\subsection{Section rings and Hochschild--Serre rings}\label{ss:A-rings}
Fix $q\ge1$. We attach to $X$ the following graded rings, graded by $m\ge0$:
\begin{itemize}
\item the polynomial ring $\Sym V_q(X)$ on $V_q(X)=H^0(X,-qK_X)$;
\item the section ring $\Gamma_q(X)=\bigoplus_{m\ge0}H^0(X,-qmK_X)$;
\item the homogeneous coordinate ring $S_q(X)=\Sym V_q(X)/I_X$, where $\mu_X\colon\Sym V_q(X)\to\Gamma_q(X)$ is multiplication of sections and $I_X=\ker\mu_X$; thus $S_q(X)=\operatorname{im}\mu_X\subset\Gamma_q(X)$;
\item the \emph{Hochschild--Serre ring} $\mathbb H_q(X)=\bigoplus_{m\ge0}\HS^{qmn,-qm}_X$, a graded subalgebra of $\HS^{\bullet,\bullet}_X$ along the ray of bidegrees $(qn,-q)$;
\item the subalgebra $\mathbb B_q(X)\subset\mathbb H_q(X)$ generated by $\mathbb H_q(X)_1=\HS^{qn,-q}_X$.
\end{itemize}
If $-qK_X$ is very ample, $|-qK_X|$ embeds $X$ in $\PP(V_q(X)^\vee)$, and $I_X$ is the homogeneous ideal of this embedding: a homogeneous polynomial of degree $m$ vanishes on $X$ exactly when its restriction, which is its image under $\mu_X$, is zero. In particular $I_X$ is saturated, $I_{X,0}=I_{X,1}=0$, and
\begin{equation}\label{eq:ProjX}
X=\Proj S_q(X)=\Proj\bigl(\Sym V_q(X)/I_X\bigr)\subset\PP(V_q(X)^\vee).
\end{equation}
No projective normality is needed: $S_q(X)$ may be smaller than $\Gamma_q(X)$.

By Lemma~\ref{lem:mult}, the maps $\theta_{qm}$ assemble into a homomorphism of graded rings
\[
\theta^{(q)}_X=\bigoplus_{m\ge0}\theta_{qm}\colon\Gamma_q(X)\longrightarrow\mathbb H_q(X).
\]

\begin{proposition}\label{prop:relations}
Assume that $\theta_q\colon V_q(X)\to\HS^{qn,-q}_X$ is bijective; by \eqref{eq:HSbound} this holds if $q\le n-2$. Then:
\begin{enumerate}[label=\textup{(\arabic*)}]
\item the elements of $\HS^{qn,-q}_X$ commute pairwise, so that multiplication defines a surjective homomorphism of graded rings $\mu_{\HS,X}\colon\Sym\HS^{qn,-q}_X\to\mathbb B_q(X)$, and the diagram
\begin{equation}\label{eq:diagsingle}
\begin{tikzcd}
\Sym V_q(X)\arrow[r,two heads,"\mu_X"]\arrow[d,"\Sym\theta_q"',"\wr"] & S_q(X)\arrow[r,hook]\arrow[d,"\theta^{(q)}_X"] & \Gamma_q(X)\arrow[d,"\theta^{(q)}_X"]\\
\Sym\HS^{qn,-q}_X\arrow[r,two heads,"\mu_{\HS,X}"] & \mathbb B_q(X)\arrow[r,hook] & \mathbb H_q(X)
\end{tikzcd}
\end{equation}
of graded rings commutes; in particular $\theta^{(q)}_X$ maps $S_q(X)$ onto $\mathbb B_q(X)$;
\item for every $m\ge1$ the following are equivalent:
\begin{enumerate}[label=(\alph*)]
\item the restriction $\theta_{qm}\colon S_q(X)_m\to\mathbb B_q(X)_m$ is an isomorphism;
\item $\ker\theta_{qm}\cap S_q(X)_m=0$;
\item $\Sym^m\theta_q$ maps $I_{X,m}$ onto $(\ker\mu_{\HS,X})_m$;
\end{enumerate}
\item the equivalent conditions of \textup{(2)} hold whenever $qm\le n-1$.
\end{enumerate}
\end{proposition}

\begin{proof}
\leavevmode
\begin{enumerate}[label=\textup{(\arabic*)}]
\item For $s,t\in V_q(X)$, Lemma~\ref{lem:mult} gives $\theta_q(s)\cdot\theta_q(t)=\theta_{2q}(st)=\theta_q(t)\cdot\theta_q(s)$; since $\theta_q$ is surjective, $\HS^{qn,-q}_X$ is commutative, and $\mu_{\HS,X}$ is defined and surjective onto $\mathbb B_q(X)$ by definition of $\mathbb B_q(X)$. The left square commutes on monomials, since by Lemma~\ref{lem:mult} and induction on $m$
\[
\theta_{qm}(s_1\cdots s_m)=\theta_q(s_1)\cdots\theta_q(s_m)\qquad(s_1,\dots,s_m\in V_q(X)),
\]
and monomials span $\Sym V_q(X)$. The right square commutes trivially. Since $\Sym\theta_q$ and $\mu_{\HS,X}$ are surjective, the left square gives $\theta^{(q)}_X(S_q(X))=\theta^{(q)}_X(\mu_X(\Sym V_q(X)))=\mu_{\HS,X}(\Sym\HS^{qn,-q}_X)=\mathbb B_q(X)$.
\item By (1), $\theta_{qm}\colon S_q(X)_m\to\mathbb B_q(X)_m$ is surjective, so (a) and (b) are equivalent. By the left square of \eqref{eq:diagsingle} and bijectivity of $\Sym^m\theta_q$,
\[
(\Sym^m\theta_q)^{-1}\bigl((\ker\mu_{\HS,X})_m\bigr)=\ker(\theta_{qm}\circ\mu_{X,m})=\mu_{X,m}^{-1}\bigl(\ker\theta_{qm}\cap S_q(X)_m\bigr),
\]
which equals $\ker\mu_{X,m}=I_{X,m}$ exactly when (b) holds. This proves the equivalence with (c).
\item By \eqref{eq:HSbound}, $\theta_{qm}$ is injective on all of $H^0(X,-qmK_X)$ for $qm\le n-1$, so (b) holds.\qedhere
\end{enumerate}
\end{proof}

\subsection{Proof of Theorem~\ref{thm:C}}\label{ss:A-proofC}
Let $X$ and $X'$ satisfy the hypotheses of Theorem~\ref{thm:C} for $q$ and $\ell$, and let $\Phi\colon\cO_X^{\perp}\xrightarrow{\sim}\cO_{X'}^{\perp}$ be an equivalence.

\emph{Step 1: dimension and the range of injectivity.} Since $\ell\ge2$ and $q\ell\le\dim X-1$, we have $q\le\dim X-1-q\le\dim X-2$, and similarly for $X'$. The line bundles $-qK_X$ and $-qK_{X'}$ are very ample, so they have nonzero sections, and Lemma~\ref{lem:dimq} gives $\dim X=\dim X'=n$. By \eqref{eq:HSbound}, $\theta_q$ is bijective on $X$ and on $X'$, and $\theta_{qm}$ is injective for $1\le m\le\ell$, since $qm\le q\ell\le n-1$. Thus Proposition~\ref{prop:relations} applies to both varieties, and its conditions (2) hold in all degrees $1\le m\le\ell$.

\emph{Step 2: the Hochschild--Serre rings.} Since $\dim X=\dim X'=n$, the isomorphism $\Phi_*$ of Section~\ref{ss:A-HS} preserves the bidegrees $(qmn,-qm)$, and by \eqref{eq:diagHSinv} it restricts to an isomorphism of graded rings
\[
\Phi_*\colon\mathbb H_q(X)\xrightarrow{\ \sim\ }\mathbb H_q(X'),
\]
which maps $\HS^{qn,-q}_X$ onto $\HS^{qn,-q}_{X'}$, hence $\mathbb B_q(X)$ onto $\mathbb B_q(X')$; moreover $\Phi_*\circ\mu_{\HS,X}=\mu_{\HS,X'}\circ\Sym(\Phi_*|_{\HS^{qn,-q}_X})$, because both sides are ring homomorphisms which agree in degree one.

\emph{Step 3: the identification of sections.} Define the linear isomorphism
\begin{equation}\label{eq:alpha}
\alpha=(\theta_q^{X'})^{-1}\circ\Phi_*\circ\theta_q^X\colon V_q(X)\xrightarrow{\ \sim\ }V_q(X').
\end{equation}
Changing the isomorphism $\Phi\mathsf S\simeq\mathsf S\Phi$ by a scalar multiplies $\alpha$ by a scalar, so the projective class of $\alpha$ is determined by $\Phi$.

\emph{Step 4: the main diagram.} Consider the diagram of graded rings
\begin{equation}\label{eq:diagmain}
\begin{tikzcd}[column sep=large]
\Sym V_q(X)\arrow[r,"\Sym\theta_q^X","\sim"']\arrow[d,two heads,"\mu_X"'] & \Sym\HS^{qn,-q}_X\arrow[r,"\Sym\Phi_*","\sim"']\arrow[d,two heads,"\mu_{\HS,X}"] & \Sym\HS^{qn,-q}_{X'}\arrow[d,two heads,"\mu_{\HS,X'}"'] & \Sym V_q(X')\arrow[l,"\Sym\theta_q^{X'}"',"\sim"]\arrow[d,two heads,"\mu_{X'}"]\\
S_q(X)\arrow[r,"\theta^{(q)}_X"] & \mathbb B_q(X)\arrow[r,"\Phi_*","\sim"'] & \mathbb B_q(X') & S_q(X')\arrow[l,"\theta^{(q)}_{X'}"']
\end{tikzcd}
\end{equation}
The left and right squares are the left squares of \eqref{eq:diagsingle} for $X$ and $X'$, and the middle square is Step~2; all three commute. By \eqref{eq:alpha}, the composite of the top row from $\Sym V_q(X)$ to $\Sym V_q(X')$ is $\Sym\alpha$. In each degree $1\le m\le\ell$, the two outer bottom arrows are isomorphisms by Proposition~\ref{prop:relations}(2)(a), and the middle one is an isomorphism. Consequently, in degrees $m\le\ell$:
\begin{enumerate}[label=(\roman*)]
\item the truncated homogeneous coordinate rings are identified: the composite of the bottom row defines isomorphisms
\[
\bar\alpha_m=(\theta_{qm}^{X'})^{-1}\circ\Phi_*\circ\theta_{qm}^X\colon S_q(X)_m\xrightarrow{\ \sim\ }S_q(X')_m\qquad(0\le m\le\ell),
\]
compatible with products whenever the degrees stay at most $\ell$, and satisfying $\bar\alpha\circ\mu_X=\mu_{X'}\circ\Sym\alpha$ in these degrees;
\item the equations are identified: since $\Sym^m\alpha$ is bijective and $\bar\alpha_m$ is injective, the relation $\bar\alpha_m\circ\mu_{X,m}=\mu_{X',m}\circ\Sym^m\alpha$ gives
\begin{equation}\label{eq:idealsmatch}
\Sym^m\alpha\,(I_{X,m})=I_{X',m}\qquad(2\le m\le\ell).
\end{equation}
Equivalently, by Proposition~\ref{prop:relations}(2)(c), the kernels in the top row are matched through the Hochschild--Serre algebras:
\[
I_{X,m}\xrightarrow[\sim]{\ \Sym^m\theta_q^X\ }(\ker\mu_{\HS,X})_m\xrightarrow[\sim]{\ \Sym^m\Phi_*\ }(\ker\mu_{\HS,X'})_m\xleftarrow[\sim]{\ \Sym^m\theta_q^{X'}\ }I_{X',m}.
\]
\end{enumerate}

\emph{Step 5: the ideals.} Let $J'\subset\Sym V_q(X')$ be the ideal generated by $I_{X',2},\dots,I_{X',\ell}$. By hypothesis, $I_X$ is generated by $I_{X,2},\dots,I_{X,\ell}$; hence, by \eqref{eq:idealsmatch}, the ring isomorphism $\Sym\alpha$ maps $I_X$ onto $J'$, and $J'\subset I_{X'}$. We obtain the commutative diagram of graded rings, with exact rows,
\begin{equation}\label{eq:diagideals}
\begin{tikzcd}
0\arrow[r] & I_X\arrow[r]\arrow[d,"\Sym\alpha","\wr"'] & \Sym V_q(X)\arrow[r]\arrow[d,"\Sym\alpha","\wr"'] & S_q(X)\arrow[r]\arrow[d,"\overline{\Sym\alpha}","\wr"'] & 0\\
0\arrow[r] & J'\arrow[r] & \Sym V_q(X')\arrow[r] & \Sym V_q(X')/J'\arrow[r]\arrow[d,two heads] & 0\\
 & & & S_q(X')=\Sym V_q(X')/I_{X'} &
\end{tikzcd}
\end{equation}
in which the induced map $\overline{\Sym\alpha}$ is an isomorphism of graded rings, and the vertical map on the right is the quotient by $I_{X'}/J'$.

\emph{Step 6: the varieties.} Apply $\Proj$ to the right column of \eqref{eq:diagideals}. Let $\psi\colon\PP(V_q(X)^\vee)\xrightarrow{\sim}\PP(V_q(X')^\vee)$ be the projective linear isomorphism induced by $\alpha$, so that a homogeneous polynomial $g$ vanishes at a point $x$ if and only if $\Sym\alpha(g)$ vanishes at $\psi(x)$. By \eqref{eq:ProjX} and \eqref{eq:diagideals} we obtain the commutative diagram
\begin{equation}\label{eq:diagProj}
\begin{tikzcd}
X=\Proj S_q(X)\arrow[r,"\psi|_X","\sim"']\arrow[d,hook] & \Proj\bigl(\Sym V_q(X')/J'\bigr)=V(J')\arrow[d,hook] & X'=\Proj S_q(X')\arrow[l,hook']\arrow[dl,hook]\\
\PP(V_q(X)^\vee)\arrow[r,"\psi","\sim"'] & \PP(V_q(X')^\vee) &
\end{tikzcd}
\end{equation}
where the top left arrow is $\Proj$ of $\overline{\Sym\alpha}^{-1}$, so $V(J')=\psi(X)$, and the arrow $X'\hookrightarrow V(J')$ is the closed immersion given by the surjection $\Sym V_q(X')/J'\twoheadrightarrow S_q(X')$. Thus $X'\subset\psi(X)$. Both are integral closed subschemes of $\PP(V_q(X')^\vee)$ of dimension $n$: $X'$ by Step~1, and $\psi(X)\simeq X$. A closed subscheme of an integral scheme of the same dimension, which is itself integral, is equal to it, so $X'=\psi(X)$, and $J'$ and $I_{X'}$ have the same saturation. Therefore $\psi$ restricts to an isomorphism
\[
f=\psi|_X\colon X\xrightarrow{\ \sim\ }X',\qquad f^*\cO_{\PP(V_q(X')^\vee)}(1)|_{X'}\simeq\cO_{\PP(V_q(X)^\vee)}(1)|_X,
\]
that is, $f$ identifies the embeddings by $|-qK_X|$ and $|-qK_{X'}|$. Since $I_X$ and $I_{X'}$ are the saturated homogeneous ideals of $X$ and $X'=\psi(X)$, we also get $\Sym\alpha(I_X)=I_{X'}$, so $J'=I_{X'}$ and the bottom vertical arrow in \eqref{eq:diagideals} is an isomorphism.

\emph{Step 7: the presentation.} By the left square of \eqref{eq:diagsingle} and Proposition~\ref{prop:relations}(2)(c), the isomorphism $\Sym\theta^X_q$ maps $I_X$, which is generated by $I_{X,2},\dots,I_{X,\ell}$, onto the ideal of $\Sym\HS^{qn,-q}_X$ generated by $(\ker\mu_{\HS,X})_m$, $2\le m\le\ell$. Hence \eqref{eq:ProjX} gives
\begin{equation}\label{eq:presentation}
X\simeq\Proj\Bigl(\Sym\HS^{qn,-q}_X\Big/\bigl((\ker\mu_{\HS,X})_m:2\le m\le\ell\bigr)\Bigr),
\end{equation}
which is the presentation in Theorem~\ref{thm:C}; here $n=\dim X$ is determined by $\cO_X^{\perp}$ by Lemma~\ref{lem:dimq}. This completes the proof of Theorem~\ref{thm:C}. \qed

\begin{remark}\label{rem:variants}
\leavevmode
\begin{enumerate}[label=(\alph*)]
\item The generation hypothesis is used only for $X$, in Step~5; for $X'$ it suffices that $-qK_{X'}$ is very ample and $q\ell\le\dim X'-1$, and Step~6 then shows that $I_{X'}$ is generated in degrees at most $\ell$.
\item The uniform bound $q\ell\le n-1$ is used only to guarantee the conditions of Proposition~\ref{prop:relations}(2) in degrees $m\le\ell$. If $\dim X=\dim X'$, $\theta_q$ is bijective on both sides, $I_X$ is generated in degrees belonging to a finite set $\mathcal D\subset\ZZ_{\ge2}$, and these conditions hold on $X$ and on $X'$ for every $m\in\mathcal D$, then the same argument, with $\mathcal D$ in place of $\{2,\dots,\ell\}$, gives the same conclusion. The conditions are automatic for $qm\le n-1$, but in general they are not easy to verify.
\item The classes used lie in the bidegrees $(qmn,-qm)$, that is, they are transformations to $\mathsf S^{-qm}[qmn]\simeq\Rot^{qm}=(\pi Ti)^{qm}$, which by \eqref{eq:Qm} involves all intermediate projections $P$. The argument makes no assumption on $-K_X$ itself, nor on projective normality.
\item For $q=1$, Theorem~\ref{thm:C} is contained in Theorem~\ref{thm:B}, but \eqref{eq:presentation} gives an explicit presentation of the anticanonical model in terms of $\HS^{n,-1}_X$ and the kernels of $\Sym^m\HS^{n,-1}_X\to\HS^{mn,-m}_X$.
\end{enumerate}
\end{remark}

\subsection{Examples}\label{ss:A-examples}

\begin{proposition}\label{prop:doublecovers}
Let $q,k\ge2$, put $n=(q+1)k-1$ and $d=qk$, let $B\subset\PP^n$ be a smooth hypersurface of degree $2d$ with equation $f$, and let $\pi\colon X\to\PP^n$ be the double cover branched along $B$, with $H=\pi^*\cO_{\PP^n}(1)$.
\begin{enumerate}[label=\textup{(\arabic*)}]
\item $X$ is a smooth Fano variety of dimension $n$ with $-K_X=kH$, and $-K_X$ is not very ample.
\item $-qK_X=dH$ is very ample, and the homogeneous ideal of $X\subset\PP(H^0(X,dH)^\vee)$ is generated by quadrics.
\item $2q\le n-1$, with equality if and only if $k=2$.
\end{enumerate}
Consequently $X$ satisfies the hypotheses of Theorem~\ref{thm:C} with $\ell=2$: if $X'$ is a smooth Fano variety with $-qK_{X'}$ very ample and $2q\le\dim X'-1$, and $\cO_{X'}^{\perp}\simeq\cO_X^{\perp}$, then $X'\simeq X$.
\end{proposition}

\begin{proof}
\leavevmode
\begin{enumerate}[label=\textup{(\arabic*)}]
\item By definition $X=\Spec_{\PP^n}(\cO\oplus\cO(-d))$, with the algebra structure given by $f\colon\cO(-d)\otimes\cO(-d)=\cO(-2d)\to\cO$; equivalently, $X=\{z^2=f(x_0,\dots,x_n)\}\subset\PP(1^{n+1},d)$, which does not contain the singular point $(0:\dots:0:1)$, since $z^2=f(0)=0$ forces $z=0$. The map $\pi$ is étale over $\PP^n\setminus B$. Near a point of $B$, choose local coordinates $(u,x_2,\dots,x_n)$ on $\PP^n$ with $B=\{u=0\}$, which is possible since $B$ is smooth; then $X$ is locally $\{z^2=u\}$, which is smooth with coordinates $(z,x_2,\dots,x_n)$. Since $B$ is reduced, $f$ is not a square, and $X$ is irreducible. The ramification divisor $R=\{z=0\}$ is the zero locus of the section $z$ of $\pi^*\cO(d)$, which vanishes to first order along it, so $\cO_X(R)\simeq\pi^*\cO(d)$, and the ramification formula gives
\[
K_X=\pi^*K_{\PP^n}+R=(d-n-1)H,\qquad-K_X=(n+1-d)H=kH.
\]
Since $\pi$ is finite, $H$ is ample, so $X$ is Fano. From $\pi_*\cO_X=\cO\oplus\cO(-d)$ we get, for all $j$,
\begin{equation}\label{eq:dcsections}
H^0(X,jH)=H^0(\PP^n,\cO(j))\oplus H^0(\PP^n,\cO(j-d))\,z.
\end{equation}
Since $q\ge2$, we have $k<d$, so $H^0(X,-K_X)=\pi^*H^0(\PP^n,\cO(k))$. The morphism defined by $|-K_X|$ is therefore the composite of $\pi$ with the $k$-th Veronese embedding; it is two-to-one over $\PP^n\setminus B$, so $-K_X$ is not very ample.
\item Let $S=\CC[x_0,\dots,x_n]$ and $S^{(d)}=\bigoplus_jS_{jd}$. By \eqref{eq:dcsections}, the section ring of $dH$ is
\[
\bigoplus_{j\ge0}H^0(X,jdH)=\bigoplus_{j\ge0}\bigl(S_{jd}\oplus S_{(j-1)d}\,z\bigr)\simeq S^{(d)}[z]/(z^2-f),
\]
with $z$ in degree one and $f\in S_{2d}=S^{(d)}_2$; the multiplication agrees because $z^2=f$ in the algebra of the cover. This ring is generated in degree one, by $S_d$ and $z$, since $S^{(d)}$ is generated by $S_d$. Since $dH$ is ample, $X$ is the $\Proj$ of this ring, and generation in degree one shows that $|dH|$ is a closed embedding, so $dH=-qK_X$ is very ample. Put $V=H^0(X,dH)=S_d\oplus\CC z$, so that $\Sym V=(\Sym S_d)[z]$. Let $J\subset\Sym S_d$ be the kernel of $\Sym S_d\to S^{(d)}$, the homogeneous ideal of the Veronese variety $\nu_d(\PP^n)$, which is generated by quadrics (classically; see \cite[Theorem~1.1]{HLMP21} for a stronger statement), and choose $F\in\Sym^2S_d$ mapping to $f$, which is possible since $S_d\cdot S_d=S_{2d}$. Then $\Sym V/J\Sym V=S^{(d)}[z]$, and dividing further by $z^2-F$ gives the section ring of $dH$. Hence the kernel of the surjection from $\Sym V$ onto the section ring is $J\Sym V+(z^2-F)$. This kernel is the homogeneous ideal of $X\subset\PP(V^\vee)$, and it is generated by quadrics.
\item $n-1-2q=(q+1)k-2-2q=(q+1)(k-2)\ge0$, with equality exactly for $k=2$.
\end{enumerate}
The final assertion follows from (1)--(3), the proof of Theorem~\ref{thm:C} with $\ell=2$, and Remark~\ref{rem:variants}(a).
\end{proof}

\begin{example}\label{ex:qone}
Let $X=Q_1\cap Q_2\cap C\subset\PP^7$ be a smooth complete intersection of two quadrics and a cubic. Then $n=4$, $-K_X=\cO_X(1)$, and the ideal of $X$ is generated by $Q_1,Q_2,C$, so Theorem~\ref{thm:C} applies with $q=1$ and $\ell=3=n-1$. Here $\dim I_{X,2}=2$ and $\dim I_{X,3}=17$, since the two quadrics have no linear syzygy, while $h^0(X,-2K_X)=34$ and $h^0(X,-3K_X)=103$; the presentation \eqref{eq:presentation} recovers $X$ from the kernels of $\Sym^m\HS^{4,-1}_X\to\HS^{4m,-m}_X$ for $m=2,3$.
\end{example}

\subsection{Quadratically defined anticanonical embeddings}\label{ss:A-quadratic}
The case $q=1$, $\ell=2$ of Theorem~\ref{thm:C} applies in every dimension $n\ge3$, and only involves the bidegrees $(n,-1)$ and $(2n,-2)$ of the Hochschild--Serre algebra.

\begin{corollary}\label{cor:quadratic}
Let $X$ be a smooth Fano variety of dimension $n\ge3$ such that $-K_X$ is very ample and the homogeneous ideal of the anticanonical embedding $X\subset\PP(H^0(X,-K_X)^\vee)$ is generated by quadrics. Put
\[
Q_X=\ker\bigl(\Sym^2\HS^{n,-1}_X\longrightarrow\HS^{2n,-2}_X\bigr).
\]
\begin{enumerate}[label=\textup{(\arabic*)}]
\item The map $\Sym^2\theta_1$ identifies the space of quadrics containing $X$ with $Q_X$, and
\[
X\simeq\Proj\bigl(\Sym\HS^{n,-1}_X/(Q_X)\bigr),
\]
that is, $X$ is the subscheme of $\PP\bigl((\HS^{n,-1}_X)^\vee\bigr)$ cut out by the quadrics in $Q_X$.
\item If $X'$ is a smooth Fano variety of dimension at least three with very ample anticanonical bundle and $\cO_{X'}^{\perp}\simeq\cO_X^{\perp}$, then $X'\simeq X$, by an isomorphism identifying the anticanonical embeddings.
\end{enumerate}
\end{corollary}

\begin{proof}
Here $q=1$, $\ell=2$ and $q\ell=2\le n-1$. Part (1) is Proposition~\ref{prop:relations}(2)(c) with $m=2$, together with \eqref{eq:presentation}. For (2), the proof of Theorem~\ref{thm:C} applies by Remark~\ref{rem:variants}(a): the generation hypothesis is only needed for $X$, and $X'$ satisfies $-K_{X'}$ very ample and $2\le\dim X'-1$.
\end{proof}

Table~\ref{tab:quadratic} lists smooth Fano threefolds to which Corollary~\ref{cor:quadratic} applies, together with the prime Fano threefolds of index one to which it does not apply, and one family in every dimension. In the table, $N=h^0(X,-K_X)-1$, and $\dim I_{X,2}$ is the number of independent quadrics containing the anticanonical model.

\begin{table}[ht]
\centering
\renewcommand{\arraystretch}{1.15}
\begin{tabular}{>{\raggedright\arraybackslash}p{6.4cm}ccc>{\raggedright\arraybackslash}p{3.3cm}}
\toprule
Fano variety $X$ & $\dim X$ & $X\subset\PP^N$ & $\dim I_{X,2}$ & Corollary~\ref{cor:quadratic}\\
\midrule
\multicolumn{5}{l}{\emph{Prime Fano threefolds of index one, genus $g$}}\\
$g=2$: sextic double solids & $3$ & --- & --- & no: $-K_X$ not very ample\\
$g=3$: double covers of a quadric threefold & $3$ & --- & --- & no: $-K_X$ not very ample\\
$g=3$: quartic threefolds & $3$ & $\PP^4$ & $0$ & no: quartic equation$^{\ast}$\\
$g=4$: complete intersections $Q\cap C\subset\PP^5$ & $3$ & $\PP^5$ & $1$ & no: cubic equation$^{\ast}$\\
$g=5$: complete intersections $Q_1\cap Q_2\cap Q_3\subset\PP^6$ & $3$ & $\PP^6$ & $3$ & yes\\
$g=6$: Gushel--Mukai threefolds, ordinary and special & $3$ & $\PP^7$ & $6$ & yes\\
$g=7$: linear sections of the spinor tenfold $\mathrm{OGr}_+(5,10)\subset\PP^{15}$ & $3$ & $\PP^8$ & $10$ & yes\\
$g=8$: linear sections of $\Gr(2,6)\subset\PP^{14}$ & $3$ & $\PP^9$ & $15$ & yes\\
$g=9$: linear sections of $\mathrm{LGr}(3,6)\subset\PP^{13}$ & $3$ & $\PP^{10}$ & $21$ & yes\\
$g=10$: linear sections of the adjoint $G_2$-variety in $\PP^{13}$ & $3$ & $\PP^{11}$ & $28$ & yes\\
$g=12$: the threefolds $V_{22}$ & $3$ & $\PP^{13}$ & $45$ & yes\\
\midrule
\multicolumn{5}{l}{\emph{Other Fano threefolds}}\\
Verra threefolds: divisors of bidegree $(2,2)$ in $\PP^2\times\PP^2$ & $3$ & $\PP^8$ & $10$ & yes\\
\midrule
\multicolumn{5}{l}{\emph{Higher dimension}}\\
complete intersections of $n$ quadrics in $\PP^{2n}$ & $n\ge3$ & $\PP^{2n}$ & $n$ & yes\\
\bottomrule
\end{tabular}

\smallskip
\begin{minipage}{\linewidth}\footnotesize
$^{\ast}$\,Theorem~\ref{thm:B} applies: the anticanonical bundle of a quartic threefold ($g=3$), and of a complete intersection $Q\cap C\subset\PP^5$ ($g=4$), is very ample. Only the quadratic presentation of Corollary~\ref{cor:quadratic} is not available, since the ideal is not generated by quadrics.
\end{minipage}
\caption{Smooth Fano varieties and Corollary~\ref{cor:quadratic}.}
\label{tab:quadratic}
\end{table}

The entries of Table~\ref{tab:quadratic} are justified as follows.
\begin{enumerate}[label=(\alph*)]
\item \emph{Prime Fano threefolds of index one.} Let $X$ be a smooth prime Fano threefold of index one and genus $g$, so that $(-K_X)^3=2g-2$ and $h^0(X,-K_X)=g+2$. The anticanonical bundle is very ample unless $X$ is hyperelliptic, that is, a sextic double solid or a double cover of a quadric threefold. If $X$ is not hyperelliptic, its anticanonical model $X_{2g-2}\subset\PP^{g+1}$ is projectively normal, and its homogeneous ideal is generated by quadrics unless $X$ is trigonal, which happens exactly for the quartic threefolds ($g=3$) and the complete intersections of a quadric and a cubic in $\PP^5$ ($g=4$); see \cite[Chapter~4]{IP99}, and \cite{SD74} for the corresponding statements for the K3 surfaces cut out by hyperplanes. The descriptions for $g\ge6$ are those of Mukai, recalled in \cite{IP99}. In the quadratically defined cases, Riemann--Roch gives $h^0(X,-2K_X)=5g$, and projective normality gives
\[
\dim I_{X,2}=\binom{g+3}{2}-5g=\frac{(g-2)(g-3)}{2}.
\]
For $g=6$ this is the description of Gushel--Mukai threefolds as intersections of six quadrics \cite{DK18}.
\item \emph{Verra threefolds.} A smooth divisor $X\subset\PP^2\times\PP^2$ of bidegree $(2,2)$ has $-K_X=\cO_X(1,1)$ and $h^0(X,-K_X)=9$, and $X\subset\PP^8$ is the intersection of the Segre variety, cut out by the nine $2\times2$ minors of a $3\times3$ matrix of coordinates, with a quadric lifting the equation of $X$. Since the Segre variety is arithmetically Cohen--Macaulay, the ideal of $X$ is generated by these ten quadrics, and $h^0(X,-2K_X)=35=45-10$.
\item \emph{Complete intersections of quadrics.} A smooth complete intersection of $n$ quadrics in $\PP^{2n}$ has dimension $n$ and $-K_X=\cO_X(2n+1-2n)=\cO_X(1)$, and by the Koszul resolution its ideal is generated by the $n$ quadrics. For $n=3$ this is the case $g=5$ above.
\end{enumerate}
In all the cases marked ``yes'', Corollary~\ref{cor:quadratic} reconstructs $X$ from $\cO_X^{\perp}$ as the intersection of the quadrics in $Q_X$, and identifies $\dim Q_X$ with the number in the fourth column.

\end{document}